\documentclass[a4paper,oneside,11pt]{article}

\usepackage[latin2]{inputenc}
\usepackage[T1]{fontenc}
\usepackage[includehead,includefoot,inner=1.2in,outer=1in,top=1in,bottom=1in]{geometry}
\usepackage{color}
\usepackage[comma,round]{natbib}
\usepackage{amsmath,amssymb,amsthm,bbm,dsfont,authblk}

\theoremstyle{plain}
\newtheorem{theorem}{Theorem}[section]
\newtheorem{lemma}[theorem]{Lemma}
\newtheorem{proposition}[theorem]{Proposition}
\newtheorem{corollary}[theorem]{Corollary}
\theoremstyle{definition}

\newtheorem{assumption}{Assumption}
\theoremstyle{remark}
\newtheorem{remark}{Remark}

\newcommand{\bone}{\ensuremath\mathbf{1}}
\newcommand{\N}{\ensuremath\mathbb{N}}
\newcommand{\Z}{\ensuremath\mathbb{Z}}

\newcommand{\R}{\ensuremath\mathbb{R}}

\newcommand{\bA}{\ensuremath\mathbf{A}}
\newcommand{\bb}{\ensuremath\mathbf{b}}
\newcommand{\bB}{\ensuremath\mathbf{B}}
\newcommand{\cD}{\ensuremath\mathcal{D}}
\newcommand{\cF}{\ensuremath\mathcal{F}}
\newcommand{\cN}{\ensuremath\mathcal{N}}
\newcommand{\bc}{\ensuremath\mathbf{c}}

\newcommand{\bd}{\ensuremath\mathbf{d}}
\newcommand{\cC}{\ensuremath\mathcal{C}}
\newcommand{\cH}{\ensuremath\mathcal{H}}
\newcommand{\bI}{\ensuremath\mathbf{I}}
\newcommand{\bJ}{\ensuremath\mathbf{J}}
\newcommand{\bm}{\ensuremath\mathbf{m}}

\newcommand{\KK}{\ensuremath\mathds{K}}
\newcommand{\bM}{\ensuremath\mathbf{M}}
\newcommand{\MM}{\ensuremath\mathds{M}}

\newcommand{\NN}{\ensuremath\mathds{N}}

\newcommand{\bx}{\ensuremath\mathbf{x}}

\newcommand{\XX}{\ensuremath\mathds{X}}
\newcommand{\cX}{\ensuremath\mathcal{X}}
\newcommand{\by}{\ensuremath\mathbf{y}}

\newcommand{\YY}{\ensuremath\mathds{Y}}
\newcommand{\cY}{\ensuremath\mathcal{Y}}
\newcommand{\ZZ}{\ensuremath\mathds{Z}}
\newcommand{\bz}{\ensuremath\mathbf{z}}
\newcommand{\bZ}{\ensuremath\mathbf{Z}}
\newcommand{\cZ}{\ensuremath\mathcal{Z}}
\newcommand{\bv}{\ensuremath\mathbf{v}}
\newcommand{\bV}{\ensuremath\mathbf{V}}
\newcommand{\VV}{\ensuremath\mathds{V}}
\newcommand{\cW}{\ensuremath\mathcal{W}}
\newcommand{\balpha}{\ensuremath\boldsymbol{\alpha}}
\newcommand{\Beta}{\ensuremath\boldsymbol{\beta}}
\newcommand{\bbeta}{\ensuremath\boldsymbol{\eta}}
\newcommand{\bmu}{\ensuremath\boldsymbol{\mu}}
\newcommand{\bxi}{\ensuremath\boldsymbol{\xi}}

\newcommand{\bnull}{\ensuremath\boldsymbol{0}}
\newcommand{\cls}{\textrm{CLS}}
\newcommand{\wcls}{\textrm{WCLS}}

\newcommand\1{\ensuremath\mathbbm{1}}

\newcommand{\DD}{ \overset{\cD}{=}}
\newcommand{\DA}{ \overset{\cD}{\longrightarrow}}
\newcommand{\PA}{ \overset{P}{\rightarrow}}

\newcommand\Cov{\ensuremath\mathrm{Cov}}

\newcommand\INAR{\ensuremath \mathrm{INAR(\mathit{p})}}
\newcommand\GINAR{\ensuremath \mathrm{GINAR(\mathit{p})}}

\date{}

\DeclareFontFamily{U}{mathx}{\hyphenchar\font45}
\DeclareFontShape{U}{mathx}{m}{n}{
      <5> <6> <7> <8> <9> <10>
      <10.95> <12> <14.4> <17.28> <20.74> <24.88>
      mathx10
      }{}
\DeclareSymbolFont{mathx}{U}{mathx}{m}{n}
\DeclareFontSubstitution{U}{mathx}{m}{n}
\DeclareMathAccent{\widecheck}{0}{mathx}{"71}
\DeclareMathAccent{\wideparen}{0}{mathx}{"75}

\title{A CUSUM Type Change Detection Test Based on Martingale Differences}
\author{Fanni Nedényi
\footnote{``This research
was realized in the frames of TÁMOP 4.2.4. A/2-11-1-2012-0001 ``National Excellence
Program -- Elaborating and operating an inland student and researcher personal support
system'' The project was subsidized by the European Union and co-financed by the
European Social Fund.
''}
}
\affil{University of Szeged}
\begin{document}
\maketitle




\section{Introduction}
Change-point detection in various stochastic models is an investigated problem of statistical analysis. If a certain model parameter changes over time then an expired estimation of that parameter could lead to false predictions concerning the behavior of the model. Therefore it is an important task to detect such changes as fast as possible.

The paper is about performing change-point detection in multitype Galton--Watson processes. 
The $p$-type Galton--Watson process $\XX_n, \, n=0,1,\dots, \, p\in \N$, is a discrete time Markov chain defined in Subsection \ref{intro} on the state space $\N^p$, where $\N$ is the set of nonnegative integers.
We test the null hypothesis $\cH_0$ that the distribution of the number of offsprings and innovations of the process does not change over time. If $\cH_0$ holds then the dynamics of the process is unchanged. %
For the main properties of Galton--Watson processes see \cite{mode}, \cite{athreya} and \cite{kaplan}. 

We define online procedures to detect changes in such models since the online tests have several advantages compared to the classical offline ones. It can be essential for the applications that in contrast to the offline method sequentiality enables us to detect changes shortly after the real time of change. 
The applicability of the procedures also demands to consider the case when the number of possible observations is limited. Therefore, besides the regular open-end procedures we also define closed-end ones.

We work under the noncontamination assumption introduced in \cite{chu} that for some $m\in \N$ there is no model change during the observations $\XX_0,\dots, \XX_m$, the so-called training sample. 
For every $n\in \N$ we reject $\cH_0$ if the related statistics $S_{m,n}=S_{m,n}(\XX_0,\dots, \XX_m,\XX_{m+1}, \dots, \XX_{m+n})$ introduced later exceeds the corresponding critical level $c$. We define two types of tests concerning the duration of the observation of the process. In case of the open-end procedure the test statistics is $\sup_{n\ge 1} S_{m,n}$  and for the closed-end one it is $\max_{1\le n\le Tm}S_{m,n}$ where $T>0$ is a constant meaning that we detect changes  based on the sample $\XX_0,\dots,\XX_{m+\lfloor mT\rfloor}$. In both cases we define the rejection time as the smallest $n\in \N$ when $S_{m,n}>c$ occurs or infinity if there is no such $n$.
In the paper we define the statistics $S_{m,n}, \, n\in \N$, and determine the related critical values.
As a special case the testing procedures are applicable to the $\GINAR$ (Generalized INteger-valued AutoRegressive) processes. The $\INAR$ (INteger-valued AutoRegressive) processes were introduced in \cite{du} as the integer-valued interpretation of the AR($p$) processes. The numbers of offsprings in an $\INAR$ model are Bernoulli distributed. By resolving this assumption we get the $\GINAR$ models having a wider applicability. The main properties, stationarity, and parameter estimators of the $\GINAR$ models are investigated in \cite{dion}. 

Change-point detection in various models is an examined problem for several years. We only mention some papers of the topic that are directly related to our paper.
In \cite{tszabo} offline procedures are defined in order to detect changes in $\INAR$ models. A large-scale simulation study is presented in \cite{szim}.
The procedures in our paper are online CUSUM-type (CUmulated SUMs) statistics motivated by the general setup of \cite{chu}. The motivations of this approach are the papers \cite{horvath} and \cite{horvath2} where open-end CUSUM-type tests are defined to perform change-point detection in their linear regression models. Furthermore, open-end and closed-end tests are also introduced in the paper \cite{kirch} to detect changes in nonlinear autoregression models.
In our paper the methods seen in the latter ones are applied to multitype Galton--Watson processes. 

The organization of the paper is the following. The main results are stated in Section \ref{main} with the proofs in Section \ref{proof}. Theorem \ref{fotetel1} leads to the definition of the open-end and closed-end sequential procedures detecting model changes in multitype Galton--Watson processes and as a special case in $\GINAR$ processes. As an application of these procedures a simulation study is detailed in Section \ref{simulation}. 

\newpage
\section{Main results}\label{main}

\subsection{Multitype Galton--Watson processes}\label{intro}
The process $\XX_n=[X_{n,1}, \dots, X_{n,p}]^\top, \, n=0,1,\dots$,  is a multitype Galton--Watson process on the state space $\N^p$  with a fixed parameter $p\in \N$ and
a random or deterministic initial vector $\XX_0$ if 
\[
\XX_n 
=\sum_{k=1}^{X_{n-1,1}} \bxi_1(n,k) + \cdots + \sum_{k=1}^{X_{n-1,p}} \bxi_p(n,k) + \bbeta(n), \qquad n\ge 1,
\]
where all the non-negative $p$-dimensional random vectors 
\begin{equation}\label{def}
\bxi_i(n,k), \ \bbeta(n), \qquad n=1,2,\dots, \qquad i=1,\dots,p, \qquad k=1,2, \dots
\end{equation}
are independent of each other and the random vectors $\{\bxi_i(n,1), \bxi_i(n,2),\dots\}$ are i.i.d for every $n=1,2,\dots$ and $i=1,\dots, p$.
We assume that the components of the vectors in (\ref{def}) are independent of each other.
For simplicity we define the $p+1$-dimensional vector
\[
\YY_n:=
\left[
\begin{array}{c}
\XX_n\\
1
\end{array}
\right]
=
\left[
\begin{array}{c}
X_{n,1}\\
\vdots\\
X_{n,p}\\
1
\end{array}
\right], \qquad n=0,1,\dots
\]

Let us consider the null hypothesis $\cH_0$ that $\{\bxi_i(1,1), \bxi_i(2,1),\dots\}$ are identically distributed for any $i=1,\dots,p$ and $\{\bbeta(1), \bbeta(2),\dots\}$ are also identically distributed meaning that the model does not change over time. Under the null hypothesis $\cH_0$ in the followings we refer to the distributions of the vectors of the number of offsprings and innovations by $\bxi_i, \ i=1,\dots,p$, and $\bbeta$ with components $\xi_{1,i},\dots, \xi_{p,i}, \ i=1,\dots,p$, and $\eta_1, \dots, \eta_p$, respectively, as they are independent of the parameters $n$ and $k$. 
{By (\ref{def}) it is clear that the random variables $\xi_{j,i}, \eta_j$ are the number of $j$-type offsprings of an $i$-type individual and the number of $j$-type innovations in a generation, respectively, where $i,j=1,\dots, p$.}
We will assume that all these components have finite second moments. Let us denote the first and second moments of the numbers of offsprings and the innovations by
\[
\mu_{i,j} := E(\xi_{i,j}),
\qquad
\mu_{i,\eta} := E(\eta_i),
\qquad
v_{i,j} := D^2(\xi_{i,j}),
\qquad
v_{i,\eta} := D^2(\eta_i).
\]
for any $i,j=1,\dots,p$.
For shorter terms we introduce the matrices
\[
\bm:=
\left[
\begin{array}{ccc}
\mu_{1,1} & \dots & \mu_{1,p}\\
\vdots& \ddots & \vdots\\
\mu_{p,1} & \dots & \mu_{p,p}\\
\end{array}
\right],
\quad
\bmu:=\left[
\begin{array}{c}
\bmu_1^\top \\
\vdots \\
\bmu_p^\top
\end{array}
\right]
:=
\left[
\begin{array}{cccc}
\mu_{1,1} & \dots & \mu_{1,p} & \mu_{1,\eta}\\
\vdots & \ddots & \vdots & \vdots\\
\mu_{p,1} & \dots & \mu_{p,p} & \mu_{p,\eta}\\
\end{array}
\right]=[\bm, E(\bbeta)],
\]
and
\[
\bV:=\left[
\begin{array}{c}
\bv_1^\top \\
\vdots \\
\bv_p^\top
\end{array}
\right]
:=
\left[
\begin{array}{cccc}
v_{1,1} & \dots & v_{1,p} & v_{1,\eta}\\
\vdots & \ddots & \vdots & \vdots\\
v_{p,1} & \dots & v_{p,p} & v_{p,\eta}\\
\end{array}
\right].
\]
In some parts of the paper we suppose that the third and fourth moments also exist. Then similarly to the definition of $\bV$ we define the matrices $\bA, \bB\in \R^{p\times(p+1)}$ of the third and fourth central moments with rows $\balpha_i^\top, \Beta_i^\top, i=1,\dots,p$, respectively.

Throughout the paper for any vector we define the $n$-th power of the vector componentwise and the norm of the vector as the Euclidean norm. 
For any matrix $\bM$ the notation $\bM^\top$ stands for the transpose of the matrix and $\varrho(\bM)$ is the spectral radius.

As we suppose that the variables have finite second moments we can consider the series of martingale differences
$
\MM_n:=\XX_n-E\big(\XX_n | \XX_{n-1}\big)$, and $\NN_n:=\MM_n^2-E \big( \MM_n^2 | \XX_{n-1}\big)$, $n=1,2,\dots
$
In Subsection \ref{mart} we show that these martingale differences are
\[
\MM_n=\XX_n-\bmu\YY_{n-1}, \qquad  \NN_n= \MM_n^2 - \bV\YY_{n-1}
=
\MM_n^2-
\left[
\begin{array}{c}
\bv_1^\top\YY_{n-1}  \\
\vdots \\
 \bv_p^\top \YY_{n-1} \\
\end{array}
\right], \qquad n=1,2,\dots
\]
Let us define the $2p$-dimensional vector $\VV_n:=[M_{n,1},N_{n,1},M_{n,2},N_{n,2},\dots,M_{n,p},N_{n,p}]^\top $  for every $n=1,2,\dots$ where $M_{n,i}$ and $N_{n,i}$ are the $i$-th elements of $\MM_n$ and $\NN_n$, respectively. 

By Theorem 1 of \cite{szucs} if the process is stable --- meaning that $\varrho(\bm)<1$ holds --- then there is a unique invariant distribution concentrated on an aperiodic positive recurrent class that the process reaches within finite steps with probability $1$ in case of any initial distribution. Theorem 3 of \cite{szucs} states that if all the random variables in (\ref{def}) have finite $r$-th moments for some $r\in  \N$ then so does the invariant distribution. As the existence of the second moments of the variables in (\ref{def}) is assumed the invariant distribution also has finite second moments. This means that $E(\widetilde{\XX}\widetilde{\XX}^\top)<\infty$ where $\widetilde{\XX}$ is a random variable with the unique invariant distribution.
The notations marked with $\widetilde{}$ always refer to the invariant distribution in the sense that if the process starts with the initial distribution meaning that $\XX_0\DD\widetilde{\XX}$ then $\widetilde{\YY}, \, \widetilde{\MM}, \, \widetilde{\NN}, \, \widetilde{\VV}$, denote the variables $\YY_0, \, \MM_1, \, \NN_1, \, \VV_1$, respectively. 
Let us define the covariance matrices
$
\widetilde\bI =\Cov(\widetilde{\MM})
$
and 
$
\widetilde\bJ =\Cov(\widetilde{\VV}).
$ By our Proposition \ref{fugg} if the proper moment conditions hold and the components of the random vectors in (\ref{def}) are independent of each other then $\widetilde{\bI}$ is diagonal and $\widetilde{\bJ}$ is block diagonal taking the forms
\[
\widetilde{\bI}=
\left[
\begin{array}{ccc}
\bv_1^\top E(\widetilde{\YY}) & \dots & 0\\
\vdots & \ddots & \vdots  \\
0 & \dots & \bv_p^\top E(\widetilde{\YY}) \\
\end{array}
\right], 
\qquad
\widetilde\bJ=
\left[
\begin{array}{cccc}
\widetilde{\bJ}_1 &\dots  & \bnull\\
\vdots  & \ddots & \vdots \\
\bnull &\dots  & \widetilde{\bJ}_p\\
\end{array}
\right],
\]
with
\[
\widetilde\bJ_i=\Cov
\left[
\begin{array}{c}
\widetilde{M}_i\\
\widetilde{N}_i\\
\end{array}
\right]=
\left[
\begin{array}{cc}
\bv_i^\top E(\widetilde{\YY}) & \balpha_i^\top E(\widetilde{\YY})\\
\balpha_i^\top E(\widetilde{\YY}) & (\Beta_i-3\bv_i^2)^\top E(\widetilde{\YY})+2\bv_i^\top E(\widetilde{\YY}\widetilde{\YY}^\top)\bv_i\\
\end{array}
\right],
\qquad 1\le i \le p,
\]
where $\widetilde{M}_i$ and $\widetilde{N}_i$ are the $i$-th components of $\widetilde{\MM}$ and $\widetilde{\NN}$, respectively.
Let $R:=\{i=1,\dots,p:\, \bv_i^\top\neq\bnull\}$ denote the set of the types that are not deterministic respect to the past.

Let us summarize the previously mentioned conditions in the following assumption.
\begin{assumption}\label{assump1}
Unless stated otherwise we assume that the multitype Galton\- --Watson process fulfills the following assumptions.
\begin{enumerate}
\item\label{assump1:1}
The process is stable meaning that $\varrho(\bm)<1$.

\item\label{assump1:3} 
The initial vector $\XX_0$ and the variables in (\ref{def}) all have finite second moments.

\item\label{assump1:2}
The components of the random vectors $\bxi_1,\dots,\bxi_p,\bbeta$ are independent of each other. 

\item\label{assump1:4}
None of the types die out. (We say that type $j=1,\dots,p$ dies out if $(\bm^n)_{j,i}=0$ for every $n \in \N$ and  every type $i=1,\dots,p$ such that $E(\eta_i)>0$.)

\item\label{assump1:6}
There exists no vector $\bc\in \R^p$, $\bc\neq \bnull$, such that $\bc^\top \bxi_i=0$ almost surely for every $i=1,\dots,p$ and $\bc^\top \bbeta$ is degenerate.
\end{enumerate}
\end{assumption}

The assumptions \ref{assump1:1} and \ref{assump1:3} result that the invariant distribution exists and has finite second moments. Assumption \ref{assump1:2} is required in order to perform the parameter estimations detailed in Subsection \ref{estim}. Assumptions \ref{assump1:4}-\ref{assump1:6} ensure that these parameter estimators exist.

The main goal of the paper is to provide sequential procedures to test the null hypothesis $\cH_0$. 
The online CUSUM-type tests can be used under the regular assumption that there is no model change in $\XX_0, \dots, \XX_m$ for some fixed $m$. This condition is called the noncontamination assumption introduced by \cite{chu} in their general paper on CUSUM-type tests.
In case of online tests asymptotical results are stated as the length of the training sample, $m+1$, converges to infinity. Let us note that under $\cH_0$ the noncontamination assumption is satisfied for every $m\in \N$.

Based on the training sample we estimate all the previously introduced objects of the process in order to define a CUSUM test on the basis of the martingale differences $\MM_n, \ \NN_n, \ n=1,2,\dots$
Let us sum up the results of the $\cls$ (Conditional Least Squares, \cite{klimko}) and $\wcls$ (Weighted Conditional Least Squares, \cite{wei}) estimations done in Subsection \ref{estim}.  By Proposition \ref{prop} the estimators exist with probability tending to $1$ as $m\to \infty$. The formulas for the $\cls$ estimators based on the training sample $\XX_0, \dots, \XX_m$ are
\begin{equation*}
\begin{split}
&
\widehat{\bmu}_m^\cls\!=\!\!\
\left[
\sum_{n=1}^m
\XX_n
\YY_{n-1}^\top 
\right]\!\!\!
\left[
\sum_{n=1}^m \YY_{n-1}\YY_{n-1}^\top 
\right]^{-1}\!\!\!\!\!\!,
\ \ \,
\widehat{\bV}_m^\cls\!=\!\!
\left[
\sum_{n=1}^m
\big(\widehat{\MM}_{m,n}^\cls\big)^{2}\YY_{n-1}^\top 
\right]\!\!\!
\left[
\sum_{n=1}^m \YY_{n-1}\YY_{n-1}^\top 
\right]^{-1}\!\!\!\!\!\!,\\
&
\widehat{\bA}_{m}^\cls\!=\!\!
\left[
\sum_{n=1}^m
(\widehat{\MM}_{m,n}^{\cls})^3\YY_{n-1}^\top 
\right]\!\!\!
\left[
\sum_{n=1}^m \YY_{n-1}\YY_{n-1}^\top 
\right]^{-1}\!\!\!\!\!\!\!,
\ \ 
\widehat{\bB}_{m}^\cls \!= \!\left[
\sum_{n=1}^m\widehat{\KK}_{m,n}^\cls\YY^\top_{n-1}\!
\right]\!\!\!
\left[
\sum_{n=1}^m \YY_{n-1}\YY_{n-1}^\top
\right]^{-1}\!\!\!\!\!\!,
\end{split}
\end{equation*}
with $\widehat{\KK}_{m,n}^\cls:=(\widehat{\MM}_{m,n}^\cls)^4
-
3(\widehat{\bV}_{m}^\cls\YY_{n-1})^2+
3(\widehat{\bV}_{m}^\cls)^{(2)}\YY_{n-1}$ and  $\widehat{\MM}_{m,n}^\cls:=\XX_n-\widehat{\bmu}_m^\cls \YY_{n-1}$  for any $n=1,2,\dots$ where $(\widehat{\bV}_{m}^\cls)^{(2)}$ is defined as $[(\widehat{\bv}_{m,1}^\cls)^{2},\dots, (\widehat{\bv}_{m,p}^\cls)^{2}]^\top$.
We also define the $\cls$ estimators
\[ \widehat{\NN}_{m,n}^\cls:=\left(
\widehat{\MM}_{m,n}^\cls
\right)^2-\widehat{\bV}
_m^\cls\YY_{n-1}, \qquad \widehat{\VV}_{m,n}^\cls:=[\widehat{M}_{m,n,1}^\cls,\widehat{N}_{m,n,1}^\cls, \dots, \widehat{M}_{m,n,p}^\cls,\widehat{N}_{m,n,p}^\cls]^\top
\]
for any $n=1,2,\dots$, where $\widehat{M}_{m,n,i}^\cls$ and $\widehat{N}_{m,n,i}^\cls$  stand for the $i$-th, $i=1,\dots, p$, component of $\widehat{\MM}_{m,n}^\cls$ and $\widehat{\NN}_{m,n}^\cls$, respectively.

To avoid bias in the estimators caused by the outstanding observations we also define the $\wcls$ estimators in Subsection \ref{estim} as the $\cls$ estimators based on the modified process 
$
\XX'_n:=
\XX_n/\sqrt{\bone^\top  \YY_{n-1}}, \ n=1,2,\dots
$
We define the weighted versions of the vectors $\MM_n,\NN_n,\VV_n$ as \[
\MM'_n:=\frac{\MM_n}{\sqrt{\bone^\top \YY_{n-1}}}, \qquad
\NN'_n:=\frac{\NN_n}{\bone^\top \YY_{n-1}}, \qquad
\VV'_n:=\big[
M'_{n,1},N'_{n,1}, \dots,M'_{n,p},N'_{n,p}
\big],
\]
for every $n=1,2,\dots$,
and the covariance matrices related to the modified process $\XX'_n, \ n=1,2,\dots$ as $\widetilde{\bI}':=\Cov(\widetilde{\MM}')$ and $\widetilde{\bJ}':=\Cov(\widetilde{\VV}')$.
We show it in Subsection \ref{estim} that the WCLS estimators of the moments based on the sample $\XX_0, \ldots, \XX_m$ are
\begin{equation*}
\begin{split}
&
\widehat{\bmu}_m^\wcls=\left[
\sum_{n=1}^m
\frac{
\XX_n
\YY_{n-1}^\top }
{\bone^\top \YY_{n-1}}
\right]
\left[
\sum_{n=1}^m 
\frac{
\YY_{n-1}\YY_{n-1}^\top }
{\bone^\top \YY_{n-1}}
\right]^{-1},
\\
&
\widehat{\bV}_m^\wcls=
\Bigg[
\sum_{n=1}^m
\frac{(
\widehat{\MM}_{m,n}^\wcls)^{2}\YY_{n-1}^\top 
}{
\bone^\top \YY_{n-1}}
\Bigg]
\left[
\sum_{n=1}^m 
\frac{
\YY_{n-1}\YY_{n-1}^\top }
{(\bone^\top \YY_{n-1})^2}
\right]^{-1},\\
&
\widehat{\bA}_{m}^\wcls=
\Bigg[
\sum_{n=1}^m
\frac{(
\widehat{\MM}_{m,n}^\wcls)^{3}\YY_{n-1}^\top 
}{
(\bone^\top \YY_{n-1})^{3/2}}
\Bigg]
\left[
\sum_{n=1}^m
\frac{
\YY_{n-1}\YY_{n-1}^\top }
{(\bone^\top \YY_{n-1})^3}
\right]^{-1},
\\
&
\widehat{\bB}_{m}^\wcls=\left[\sum_{n=1}^m
\frac{\widehat{\KK}_{m,n}^\wcls\YY_{n-1}^\top}{(\bone^\top \YY_{n-1})^2}
\right]
\left[\sum_{n=1}^m
\frac{\YY_{n-1}\YY_{n-1}^\top}{(\bone^\top \YY_{n-1})^4}
\right]^{-1},
\end{split}
\end{equation*}
with $\widehat{\MM}_{m,n}^\wcls:=\XX'_n-\widehat{\bmu}_m^\wcls \YY_{n-1}/\sqrt{\bone^\top \YY_{n-1}}$ and 
\[
\widehat{\KK}_{m,n}^\wcls:=(\widehat{\MM}_{m,n}^\wcls)^4
-3\frac{(\widehat{\bV}_m^\wcls\YY_{n-1})^2}{(\bone^\top  \YY_{n-1})^2}+3
\frac{(\widehat{\bV}^\wcls_m)^{(2)}\YY_{n-1}}{(\bone^\top \YY_{n-1})^2}, \qquad n=1,2,\dots,
\]
where $(\widehat{\bV}_{m}^\wcls)^{(2)}$ is defined as $[(\widehat{\bv}_{m,1}^\wcls)^{2},\dots, (\widehat{\bv}_{m,p}^\wcls)^{2}]^\top$.
We also define the $\wcls$ estimators
\[ \widehat{\NN}_{m,n}^\wcls:=\left(
\widehat{\MM}_{m,n}^\wcls
\right)^2-\widehat{\bV}
_m^\wcls\frac{\YY_{n-1}}{\bone^\top \YY_{n-1}}, \qquad n=1,2,\dots
\]
and
\[
 \widehat{\VV}_{m,n}^\wcls:=[\widehat{M}_{m,n,1}^\wcls,\widehat{N}_{m,n,1}^\wcls, \dots, \widehat{M}_{m,n,p}^\wcls,\widehat{N}_{m,n,p}^\wcls]^\top, \qquad n=1,2,\dots
\]
Let us apply the notations 
\[
\overline{\YY}_m^{(\kappa)}:=\frac{1}{m} \sum_{n=1}^m \frac{\YY_{n-1}}{(\bone ^\top \YY_{n-1})^{\kappa/2}}, \qquad
\overline{\overline{\YY}}_m^{(\kappa)}:= \frac{1}{m}\sum_{n=1}^m\frac{\YY_{n-1}\YY_{n-1}^\top}{(\bone ^\top \YY_{n-1})^{\kappa/2}},
\qquad
\kappa\ge 0,
\]
and define the $\cls$  estimators of the matrices $\widetilde{\bI}$ and $\widetilde{\bJ}$ by
\[ 
\widehat{\bI}_m^\cls:=\left[
\begin{array}{ccc}
(\widehat{\bv}_{m,1}^\cls)^\top \overline{\YY}_m^{(0)} & \dots & 0\\
\vdots& \ddots & \vdots\\
0 &  \dots & (\widehat{\bv}_{m,p}^\cls)^\top \overline{\YY}_m^{(0)} \\
\end{array}
\right],
\qquad
\widehat{\bJ}_m^\cls:=
\left[
\begin{array}{ccc}
\widehat{\bJ}_{m,1}^\cls & \dots & \bnull\\
\vdots& \ddots & \vdots\\
\bnull &  \dots & \widehat{\bJ}_{m,p}^\cls \\
\end{array}
\right],
\]
where
\[
\widehat{\bJ}_{m,i}^\cls:=\left[
\begin{array}{cc}
(\widehat{\bv}_{m,i}^\cls)^\top \, \overline{\YY}_m^{(0)} & (\widehat{\balpha}_{m,i}^\cls)^\top \overline{\YY}_m^{(0)} \\
(\widehat{\balpha}_{m,i}^\cls)^\top  \overline{\YY}_m^{(0)} &   \left[\widehat{\Beta}_{m,i}^\cls-3(\widehat{\bv}_{m,i}^\cls)^2\right]^\top \overline{\YY}_m^{(0)} +2(\widehat{\bv}_{m,i}^\cls)^\top \,\overline{\overline{\YY}}_m^{(0)}  \widehat{\bv}_{m,i}^\cls
\end{array}
\right]
\]
for every $i=1, \dots, p$.
Similarly, the $\wcls$ estimators of $\widetilde{\bI}',$ and $\widetilde{\bJ}'$ are
\[
\widehat{\bI}_m^\wcls\!:=\!\left[
\begin{array}{ccc}
(\widehat{\bv}_{m,1}^\wcls)^\top \overline{\YY}_m^{(2)}  & \dots & 0\\
\vdots& \ddots & \vdots\\
0 &  \dots & (\widehat{\bv}_{m,p}^\wcls)^\top \overline{\YY}_m^{(2)} \\
\end{array}
\right], 
\ \
\widehat{\bJ}_m^\wcls\!:=\!
\left[
\begin{array}{ccc}
\widehat{\bJ}_{m,1}^\wcls & \dots & \bnull\\
\vdots& \ddots & \vdots\\
\bnull &  \dots & \widehat{\bJ}_{m,p}^\wcls \\
\end{array}
\right],
\]
where the blocks of the block diagonal matrix $\widehat{\bJ}_{m}^\wcls$ are
\[
\widehat{\bJ}_{m,i}^\wcls:=\left[
\begin{array}{ccc}
(\widehat{\bv}_{m,i}^\wcls)^\top  \overline{\YY}_m^{(2)}& & (\widehat{\balpha}_{m,i}^\wcls)^\top  \overline{\YY}_m^{(3)}\\
 (\widehat{\balpha}_{m,i}^\wcls)^\top  \overline{\YY}_m^{(3)}&  &
\left[\widehat{\Beta}_{m,i}^\wcls-3\left(\widehat
{\bv}_{m,i}^\wcls\right)^2\right]^{\! \top} \!\! \overline{\YY}_m^{(4)}+
 2(\widehat{\bv}_{m,i}^\wcls)^\top  \overline{\overline{\YY}}_m^{(4)}\widehat{\bv}_{m,i}^\wcls
\end{array}
\right]
\]
for any $ i=1,\dots,p$.
Let us define the function
\begin{equation*}\label{g}
g_{\gamma}(m,k) := \sqrt{m}\left(1+\frac{k}{m}\right) \left( \frac{k}{m+k}\right)^{\gamma}, \qquad m, k\in \N,\qquad 0  \le \gamma < \frac{1}{2}.
\end{equation*}
We introduce for any $m\in \N$ the processes
\[
\widehat{\cY}^\cls_m(t):=
\frac{
\sum_{n=m+1}^{m+\lfloor tm\rfloor} \widehat{\MM}^\cls_{m,n} }{\sqrt{m} \left(1+\frac{\lfloor tm \rfloor}{m}\right) 
\left(\frac{\lfloor tm\rfloor}{m+\lfloor tm\rfloor}\right)^\gamma}, \hspace{1.5cm}
\cY(t):=
\frac{\widetilde{\bI}^{1/2}\cW(\frac{t}{1+t}) }{(\frac{t}{1+t})^\gamma}, \qquad t\ge 0,
\]
where $\cW(t), \, t\ge 0$, is a $p$-dimensional standard Wiener process and similary
\[
\widehat{\cZ}^\cls_m(t):=
\frac{
\sum_{n=m+1}^{m+\lfloor tm\rfloor} \widehat{\VV}^\cls_{m,n} }{\sqrt{m} \left(1+\frac{\lfloor tm \rfloor}{m}\right) 
\left(\frac{\lfloor tm\rfloor}{m+\lfloor tm\rfloor}\right)^\gamma}, \hspace{1.5cm}
\cZ(t):=
\frac{\widetilde{\bJ}^{1/2}\cW'(\frac{t}{1+t}) }{(\frac{t}{1+t})^\gamma}, \qquad t\ge 0,
\]
where $\cW'(t), \, t\ge 0$, is a $2p$-dimensional standard Wiener process. We define the processes $\widehat{\cY}^\wcls_m(t)$, $\widehat{\cZ}^\wcls_m(t)$, $t\ge 0$, similarly by replacing the $\cls$ estimators with the $\wcls$ ones and the matrices $\widetilde{\bI},\, \widetilde{\bJ}$ with $\widetilde{\bI}',\, \widetilde{\bJ}'$, respectively.

These processes are the elements of $\cD^p[0,\infty)$ and $\cD^{2p}[0,\infty)$  that are the $p$ and $2p$ dimensional Skorohod spaces of the $p$ and $2p$ dimensional vector-valued c\`adl\`ag functions defined on $[0,\infty)$, respectively. For more detailes on these Skorohod spaces see Chapter VI of \cite{jacod}. A detailed investigation of $\cD[0,\infty)$ is presented in Section 16 of \cite{billingsley}.
\begin{theorem}\label{fotetel1}
The following convergences hold under $\cH_0$ and Assumption \ref{assump1}.
\begin{enumerate}
\item\label{fotetel1:1}
If for some $\varepsilon>0$ the $(4+\varepsilon)$-th and  $(2+\varepsilon)$-th moments of the variables in (\ref{def}) are finite then $\widehat{\cY}^\cls_m\DA \cY$ and $\widehat{\cY}^\wcls_m\DA \cY$, respectively, in the Skorohod space $\cD^p[0,\infty)$ as $m\to \infty$. 
\item\label{fotetel1:2}
If for some $\varepsilon>0$ the $(6+\varepsilon)$-th and fourth moments of  the variables in (\ref{def}) are finite then $\widehat{\cZ}^\cls_m\DA \cZ$ and $\widehat{\cZ}^\wcls_m\DA \cZ$, respectively, in the Skorohod space $\cD^{2p}[0,\infty)$ as $m\to \infty$.
\end{enumerate}
\end{theorem}
\begin{remark}\label{remark1}
Note, that as a consequence of Theorem \ref{fotetel1} for any measurable function $\psi: \cD^p[0,\infty)\to\R$ that is continous on the subspace $\cC^p[0,\infty)$ it holds that $\psi(\widehat{\cY}^\cls_m)\DA \psi(\cY)$ and $\psi(\widehat{\cY}_m^\wcls) \DA \psi(\cY)$ as $m\to \infty$ under the moment conditions given in Theorem \ref{fotetel1}, respectively.
Therefore, under the same conditions if $c_\alpha \in \R$ is a continuity point of the distribution function of  $\psi(\cY)$ then
\[
P\big(
\psi(\widehat{\cY}_m^\cls)>c_\alpha
\big)\to
P\big(
\psi(\cY)>c_\alpha
\big), \qquad
P\big(
\psi(\widehat{\cY}_m^\wcls)>c_\alpha
\big)\to
P\big(
\psi(\cY)>c_\alpha
\big),
\]
as $m\to \infty$.
Similar results hold for the processes $\widehat{\cZ}_m^\cls$, $\widehat{\cZ}_m^\wcls$.
By choosing such $\psi$ functions one can define test statistics where the proper $c_\alpha$ values are critical values with asymptotically $\alpha$ significance level. In the next subsection we show concrete examples for $\psi$ functions and for a simple choice we also examine the power of the related test.
\end{remark}

In the following proposition we examine the invertibility of the matrices $\widetilde{\bI}=\Cov(\widetilde{\MM})$ and  $\widetilde{\bI}'=\Cov(\widetilde{\MM'})$ that are diagonal as \ref{assump1:2} of Assumption \ref{assump1} holds. Note that diagonal matrices are invertible if all their diagonal elements are non-degenerate.
\begin{proposition}\label{prop:mart kovariancia} Let $\XX_n, \ n=0,1,\dots$, be
a Galton--Watson process satisfying \ref{assump1:1}-\ref{assump1:4} of Assumption \ref{assump1}. Then
$D^2(\widetilde M_i)=0$ and $D^2(\widetilde M_i')=0$ if and only if $\bv_i^\top=\bnull$. As a consequence the matrices $\widetilde{\bI}$ and $\widetilde{\bI}'$ are invertible exactly if $\bv_i^\top\neq\bnull$ for every $i=1,\dots,p.$
\end{proposition}
\begin{proof}
Computing the expected value we get that
\[
D^2(\widetilde M_i) = E(\widetilde M_i^2) = E\big( E[\widetilde M_i^2 \mid \widetilde\XX_0] \big)
= E( \bv_i^\top \widetilde\YY_0) = \bv_i^\top E(\widetilde\YY_0),
\]
where all elements of $E(\widetilde\YY_0)$ are strictly positive by \ref{assump1:4} of Assumption \ref{assump1}.
Similarly,
\[
D^2(\widetilde M_i') = E\left( \frac{\widetilde{M}_{i}^2}{\1^\top \widetilde{\YY}_0}\right) =  E\left(E\left[ \frac{\widetilde{M}_{i}^2}{\1^\top \widetilde{\YY}_0}\ \bigg| \ \widetilde{\XX}_0\right]\right) 
= E\left(\bv_i^\top \frac{\widetilde{\YY}_0}{\1^\top \widetilde{\YY}_0}\right)  =\bv_i^\top E\left( \frac{\widetilde{\YY}_0}{\1^\top \widetilde{\YY}_0}\right).
\]
This completes the proof.
\end{proof}
No general, satisfactory condition has been found to provide the invertibility of the matrices $\widetilde{\bJ}$ and $\widetilde{\bJ}'$. One can check the invertibility of these block diagonal matrices for the concrete model by showing that all their blocks in the diagonal are invertible.

\begin{remark}\label{fotetelalt}
If the matrices $\widetilde{\bI}$ and $\widetilde{\bJ}$ are invertible, $\widehat{\bI}_m^\cls\to \widetilde{\bI}$ and $\widehat{\bJ}_m^\cls\to \widetilde{\bJ}$ almost surely as $m\to \infty$ --- that follow under the proper moment conditions --- then 
under the conditions of Theorem \ref{fotetel1} it holds that
$
(
\widehat{\bI}_m^\cls
)^{-1/2}\widehat{\cY}_m^\cls \DA \widetilde{\bI}^{-1/2}\cY$ and $ (
\widehat{\bJ}_m^\cls
)^{-1/2}\widehat{\cZ}_m^\cls \DA \widetilde{\bJ}^{-1/2}\cZ
$,
respectively, as $m\to \infty$.
Similar arguments hold for the $\wcls$ estimators.
\end{remark}
\begin{remark}\label{reduced}
In case the covariance matrix $\widetilde{\bI}$ is degenerate we can consider the reduced, $|R|$-dimensional process $\widehat{\cY}_m^\cls |_{R}$ containing only those components of the process $\widehat{\cY}_m^\cls$ whose indices are in $R=\{i=1,\dots,p: \bv_i\neq \bnull\}$. Let $\widetilde{\bI}|_R, \, \widehat{\bI}_m^\cls|_R\in \R^{|R|\times |R|}$ be the related covariance matrices and their estimators, the reductions of $\widetilde{\bI}$ and $\widehat{\bI}_m^\cls$, respectively, consisting of the rows and columns with indices in $R$.
By Proposition \ref{prop:mart kovariancia} the reduced matrix $\widetilde{\bI}|_R\in \R^{|R|\times |R|}$ is invertible. 
Similar reduction is possible for the processes $\widehat{\cZ}_m^\cls$ and $\widehat{\cZ}_m^\wcls$ by excluding additional components.
By Remark \ref{fotetelalt} this means that $
(
 \widehat{\bI}_m^\cls|_R
)^{-1/2}\widehat{\cY}_m^\cls|_R \DA \widetilde{\bI}|_R^{-1/2}\cY|_R$ as $m\to \infty$.
Similar arguments hold for the other processes as well.

An application of these reductions can be seen in Subsection \ref{ginar} for the $\GINAR$ processes.

\end{remark}

\subsection{Test statistics and alternative hypothesis}\label{alternat}

In the previous subsection we showed that certain CUSUM-type processes converge in distribution. Now we show that applying supremum type functions to these processes we develop the testing procedures described in the Introduction.
Let us introduce some $\psi$ functions to define test statistics. We only discuss the functions concerning the process $\widehat{\cY}_m^\cls$, although they can be extended to the processes $\widehat{\cY}_m^\wcls,\, \widehat{\cZ}_m^\cls$, and $\widehat{\cZ}_m^\wcls$ as well. Fix the parameter $T\in(0,\infty]$ and recall that our aim is to detect changes based on the sample $\XX_0,\dots,\XX_{m+\lfloor mT\rfloor}$. We assume that the covariance matrix $\widetilde{\bI}$ is invertible meaning that $R=\{1,\dots,p\}$ by Proposition \ref{prop:mart kovariancia}. Otherwise, throught this subsection consider the reduction of the process defined in Remark \ref{reduced}. 
First, we define the function
\[
\psi_T^{(1)}(x):=\sup_{0\le t \le T} \|x(t)\|, \qquad x\in \cD^p[0,\infty).
\]
If $\widetilde{\bI}$ is invertible then by Remark \ref{fotetelalt} applying this function to $\big(\widehat{\bI}_m^\cls\big)^{-1/2}\widehat{\cY}_m^\cls$ we get that
\begin{equation*}
\begin{split}
&\psi_T^{(1)}\Big(\big(\widehat{\bI}_m^\cls\big)^{-1/2}\widehat{\cY}_m^\cls\Big)=\sup_{1\le k\le Tm} \frac{\|\big(\widehat{\bI}_m^\cls\big)^{-1/2}\sum_{n=m+1}^{m+k} \widehat{\MM}_{m,n}^\cls \|}{g_\gamma(m,k)}\DA
\sup_{0\le t \le T} \frac{\left\|
\cW\left(
\frac{t}{1+t}
\right)
\right\|}{\left(\frac{t}{1+t}\right)^\gamma}\\
&
\DD
\sup_{0\le t \le 1}\frac{\left|
\left|
\cW\left( \frac{T}{1+T} t
\right)
\right|
\right|
}{\left(
\frac{T}{1+T} t
\right)^\gamma}
\DD
\left(\frac{T}{1+T}\right)^{1/2-\gamma}\sup_{0\le t \le 1} \frac{\|\cW(t)\|}{t^\gamma}, \qquad m\in \N,
\end{split}
\end{equation*} 
where the alteration of the limit distribution can be verified by checking that the covariance functions of the two Gaussian processes are the same.
For $T<\infty$ we get the convergence in distribution that the closed-end, and for $T=\infty$ the one that the open-end procedure is based on. (Let us define the expression $T/(1+T)$ as $1$ in case of $T=\infty$.)
The difficulty is that there is no theoretical result describing the limit distribution if the dimension of the Wiener process is  greater than $1$. Although, in \cite{horvath} the critical values are determined for the one-dimensional case of the limit disribution. Therefore, in the followings we apply functions that reduce the dimension of the Wiener process enabling us to use the simulated critical values in \cite{horvath}.

Therefore, we consider a constant vector $\bc\in \R^p$ and the function
\[
\psi^{(2)}_T(x):=\sup_{0\le t \le T} | \bc^\top x(t)|, \qquad x\in \cD^p[0,\infty).
\]
Assuming that $\widetilde{\bI}^{-1/2}$ exists we have that
\begin{equation*}
\begin{split}
\psi^{(2)}_T&\Big(\big(\widehat{\bI}_m^\cls\big)^{-1/2}\widehat{\cY}_m^\cls\Big)=\sup_{1\le k\le Tm} \frac{|\bc^\top \big(\widehat{\bI}_m^\cls\big)^{-1/2}\sum_{n=m+1}^{m+k} \widehat{\MM}_{m,n}^\cls |}{g_\gamma(m,k)}\DA\sup_{0\le t \le T} \frac{\left|\bc^\top
\cW\left(
\frac{t}{1+t}
\right)
\right|}{\left(\frac{t}{1+t}\right)^\gamma}\\
&
\DD\left(\frac{T}{1+T}\right)^{1/2-\gamma}\sup_{0\le t \le 1} \frac{|\bc^\top\cW(t)|}{t^\gamma}
\DD\left(\frac{T}{1+T}\right)^{1/2-\gamma}\|\bc\|\sup_{0\le t \le 1} \frac{ |W(t)|}{t^\gamma}
, \qquad m\to \infty,
\end{split}
\end{equation*}
where $W(t),\, t\ge 0$, is a one-dimensional standard Wiener process.

Consider the function
\[
\psi^{(3)}_{T}(x):=\sup_{0\le t \le T} \max_{1\le i\le p} |x_i(t)|, \qquad x\in \cD^p[0,\infty).
\]
Let $a_i=a_i(m)$ denote the $i$-th diagonal element of the diagonal matrix $(\widehat{\bI}_m^\cls)^{-1/2}$ and $W_i$ the $i$-th component of $\cW$ where $i=1,\dots,p$.  In the simulation study we apply this function to the process resulting
\begin{equation*}
\begin{split}
&\psi^{(3)}_{T}\Big(\big(\widehat{\bI}_m^\cls\big)^{-1/2}\widehat{\cY}_m^\cls\Big)=\sup_{1\le k\le Tm} \max_{1\le i\le p} \frac{| a_i \sum_{n=m+1}^{m+k} \widehat{M}_{m,n,i}^\cls |}{g_\gamma(m,k)}\\
&\DA\sup_{0\le t \le T} \max_{1\le i\le p} \frac{\left|
W_i\left(
\frac{t}{1+t}
\right)
\right|}{\left(\frac{t}{1+t}\right)^\gamma}
\DD \left(\frac{T}{1+T}\right)^{1/2-\gamma}\sup_{0\le t \le 1} \max_{1\le i\le p}\frac{|W_i(t)|}{t^\gamma}, \qquad m\to \infty,
\end{split}
\end{equation*}
where $\widehat{M}_{m,n,i}^\cls$ is the $i$-th component of $\widehat{\MM}_{m,n}^\cls$.
This means that for any $c \in \R$ we have
\begin{equation*}
\begin{split}
&
P\left(
\sup_{1\le k\le Tm} \max_{1\le i\le p} \frac{| a_i \sum_{n=m+1}^{m+k} \widehat{M}_{m,n,i}^\cls |}{g_\gamma(m,k)}> c
\right)
\to 
P
\left(
\left(
\frac{T}{1+T}
\right)^{1/2-\gamma} \!\!\!
\sup_{0\le t \le 1} \max_{1\le i\le p} \frac{|W_i(t)|}{t^\gamma}>c
\right)\\
&
=
1-\left(1-
P\left(
\left(
\frac{T}{1+T}
\right)^{1/2-\gamma} \sup_{0\le t \le 1} \frac{|W_1(t)|}{t^\gamma}> c
\right)\right)^{p}, \qquad m\to \infty.
\end{split}
\end{equation*}
Let us note that if we apply the function to the reduced process $\widehat{\cY}_m^\cls|_R$ then the exponent $p$ is replaced by $|R|$.

We are going to examine the power of the test we get by applying the function $\psi^{(1)}_T$. Let us note that similar results can be achieved for the other functions as well.
We consider the alternative hypothesis $\cH_A$ that for  an index $k^*=k^*(m)\in \N$ the dynamics of the process  $\XX_n, \ n=0,1,\dots$, is unchanged until the $(m+k^*)$-th step when it switches to another dynamics but there is no change after that. This means that for any $i=1, \dots, p$ the random vectors $\{\bxi_i(1,1), \dots, \bxi_i(m+k^*-1,1)\}$ are i.i.d. and $\{\bxi_i(m+k^*,1), \dots\}$ are i.i.d. and similary $\{\bbeta(1), \dots, \bbeta(m+k^*-1)\}$ are i.i.d. and $\{\bbeta(m+k^*),\dots\}$ are i.i.d. Furthermore, we assume that the dynamics of the process changes in such a way that even the matrices of the expected values before the change, $\bmu_0$,  and after it, $\bmu_*$,  differ from each other.

The following two results are motivated by the similar theorems of \cite{horvath} and \cite{horvath2} stated for their linear regression models.

\begin{theorem}\label{alt}
Assume that the process satisfies $\cH_A$ and Assumption \ref{assump1} before and also after the change.
If for some $\varepsilon>0$ the $(4+\varepsilon)$-th moments of the random variables in (\ref{def}) are finite then 
\[
\sup_{k\ge 1} \frac{\left\|\sum_{n=m+1}^{m+k} \widehat{\MM}_{m,n}^\cls\right\|}{g_\gamma(m,k)}\PA \infty, \qquad m\to \infty.
\]
It is a direct consequence that the related tests are strongly consistent.
Also, the same result holds for the $\wcls$ estimators with the lower moment condition that the $(2+\varepsilon)$-th moments are finite for some $\varepsilon>0$.
\end{theorem}
In the next propositions we examine the time of rejection under the alternative hypothesis $\cH_A$ with significance level $\alpha$ and related critial value $x_\alpha$. Let us define $\tau_{m,\ell}\in \{1,\dots, \infty\}$ as the time of the first rejection after the $(m+\ell)$-th observation.
Precisely,
\[
\tau_{m,\ell}^\cls:=\inf \Bigg\{k\ge\ell: 
\left\|
\left(
\widehat{\bI}_m^\cls\right)^{-1/2} \sum_{n=m+1}^{m+k} \widehat{\MM}_{m,n}^\cls
\right\|> x_\alpha g_\gamma(m,k)
\Bigg\},
\]
and  we define $\tau_{m,\ell}^\wcls$ similarly by replacing the $\cls$ estimators with the $\wcls$ ones.
\begin{proposition}\label{alt3}
If the conditions of Theorem \ref{alt} hold,  for some $\theta>0$ and $b\ge0$ the time of change has the form $k^*=\lfloor \theta m^b\rfloor$ and for some $\varepsilon>0$ the $(4+\varepsilon)$-th moments of the number of offsprings and the innovations are finite then the following statements hold. 
\begin{enumerate}
\item\label{alt3:1}
If $0\le b<(1-2\gamma)/(2-2\gamma)$ then $\tau_{m,k^*}^\cls-k^*=O_P\left(m^{(1-2\gamma)/(2-2\gamma)}\right)$.
\item\label{alt3:2}
If $(1-2\gamma)/(2-2\gamma)\le b< 1$ then $\tau_{m,k^*}^\cls-k^*=O_P\left(m^{1/2-\gamma(1-b)}\right)$.
\item\label{alt3:3}
If $1 \le b<\infty$ then $\tau_{m,k^*}^\cls-k^*=O_P\left(m^{b-1/2}\right)$.
\end{enumerate}

Similar statements hold for $\tau_{m,k^*}^\wcls$ with lower moment condition, namely if for some $\varepsilon>0$ the $(2+\varepsilon)$-th moments exist.
\end{proposition}
Aside from the testing we would also like to estimate the time of change. We can do so by taking the smallest $n\in \N$ such that the statistics $S_{m,n}$ exceeds the corresponding critical level $c$. This means that our estimator of the time of change is $\tau_{m,1}^\cls$ or $\tau_{m,1}^\wcls$. Similarly, $\tau_{m,k^*}^\cls$ and $\tau_{m,k^*}^\wcls$ are the smallest $n$ where $S_{m,n}>c$ after the real time of change. Let us note that the previous proposition concerns these times, although there could be a false alarm occuring before the change. In the next proposition the probability of such a false alarm is examined.

\begin{proposition}\label{alt4}
Under the conditions of Proposition \ref{alt3} the following statements hold.
\begin{enumerate}
\item\label{alt4:1}
If $0\le b <1$ then $P\big(\tau_{m,1}^\cls<k^*\big)\to 0$ as $m\to \infty$.
\item\label{alt4:2}
If $1\le b <\infty$ then $P\big(\tau_{m,1}^\cls<k^*\big)\to c$ as $m\to \infty$ where $c\in (0,\alpha]$.
\end{enumerate}
The statements also hold for $\tau_{m,1}^\wcls$ under lower moment conditions,  if for some $\varepsilon>0$ the $(2+\varepsilon)$-th moments exist.
\end{proposition}
\begin{corollary}
As a consequence of  Proposition \ref{alt4}  the statements \ref{alt3:1} and \ref{alt3:2} of Proposition \ref{alt3} also hold by replacing $\tau_{m,k^*}$ with $\tau_{m,1}$. 
\end{corollary}

\subsection{GINAR($p$) processes}\label{ginar}

The $\GINAR$ process is a sequence $Z_n, \ n= -p+1,-p+2,\dots$, on the state space $\N$ with deterministic or random initial values $Z_{-p+1},\dots, Z_0$ and
\[
Z_n=\sum_{k=1}^{Z_{n-1}}\zeta_1(n,k)+\dots+ \sum_{k=1}^{Z_{n-p}}\zeta_p(n,k)+\eta(n), \qquad n=1,2,\dots,
\]
where 
\begin{equation}\label{def2}
\zeta_i(n,k),\, \eta(n), \qquad n=1,2,\dots,\qquad i=1,\dots,p,\qquad k=1,2,\dots
\end{equation}
are independent of each other and the sequence $\zeta_i(n,k), \ k=1,2,\dots$ consists of i.i.d. non-negative integer-valued random variables for any $i=1,\dots,p$ and $n=1,2,\dots$ Furthermore,  $\eta(n),\ n=1,2,\dots$,  is the sequence of the independent, non-negative integer-valued innovations, and all these sequences are independent of each other. We also assume that $E(\zeta_p(n,1))>0$ for every $n\in \N$.
The $\INAR$ process introduced by \cite{du} is the non-negative integer-valued analogous of the AR($p$) process. The numbers of offsprings are Bernoulli distributed with parameters $\alpha_1,\dots, \alpha_p\in[0,1]$.  The connection between the two models can be seen by computing the conditional expected values of the $\INAR$ process to the generated filtration, $\cF_n, n\in \N$. We have that
\[
E[Z_{n} \mid \cF_{n-1}]=\alpha_1 Z_{n-1}+\alpha_2 Z_{n-2}+\dots+ \alpha_p Z_{n-p}+E(\eta(n)), \qquad n=1,2,\dots,
\]
that is the same as the one of the AR($p$) process.
In \cite{dion} the process is discussed in the case where the numbers of offsprings are generally distributed. Main properties and the stationarity of the process are investigated and parameter estimations are also given in the paper. Independently, the $\INAR$ process was examined by \cite{barczy}. An offline procedure is presented to detect changes in $\INAR$ models in \cite{tszabo}.

The GINAR($p$) process is embedded in the multitype Galton--Watson process $\XX_n=[Z_n,Z_{n-1},\dots,Z_{n-p+1}]^\top$, $n=0,1,\dots$, with the corresponding vectors
\[
\bxi_1(n,k)=
\left[
\begin{array}{c}
\zeta_1(n,k) \\
1\\
0\\
\vdots \\
0
\end{array}
\right], 
\qquad
\bxi_2(n,k)=
\left[
\begin{array}{c}
\zeta_2(n,k) \\
0\\
1\\
\vdots \\
0
\end{array}
\right],
\qquad
\bxi_p(n,k)=
\left[
\begin{array}{c}
\zeta_p(n,k) \\
0\\
0\\
\vdots \\
0
\end{array}
\right]
\]
for any $k,n\in \N$ 
and the vector of innovations is $\bbeta(n)=[\eta(n), 0, \dots, 0]^\top$.
In case of the $\GINAR$ process the $\cH_0$ null hypothesis introduced in Subsection \ref{intro} holds exactly if the random variables $\{\zeta_i(1,1), \zeta_i(2,1), \dots\}$ are identically distributed for any $i=1,\dots,p$ and $\{\eta(1), \eta(2),\dots\}$ are also identically distributed. The corresponding matrices $\bmu$ and $\bV$ defined in Subsection \ref{intro} are
\[
\bmu\!=\!
\left[
\begin{array}{ccccc}
E(\zeta_1)&  \cdots & E(\zeta_{p-1}) & E(\zeta_p) & E(\eta) \\
1 		    &  \cdots &  0	&0 	    & 0 \\
\vdots      & \ddots &  \vdots    &\vdots  &\vdots\\
0 		    &  \cdots &  1	&0	    & 0 \\
\end{array}
\right],
\ \
\bV\!=\!
\left[
\begin{array}{cccc}
D^2(\zeta_1)&  \cdots & D^2(\zeta_p) & D^2(\eta) \\
0 		    &  \cdots &  0		    & 0 \\
\vdots      & \ddots &  \vdots      &\vdots\\
0 		    &  \cdots &  0		    & 0 \\
\end{array}
\right].
\]
We distinguish three cases of the $\GINAR$ process $Z_n, \ n= -p+1,-p+2,\dots$ It is
\[
\left.
\begin{array}{c}
\textrm{subcritical} \\
\textrm{critical} \\
\textrm{supercritical} \\
\end{array}
\right\}
\Longleftrightarrow
\left\{
\begin{array}{c}
\mu_{1,1}+\cdots+\mu_{1,p}<1 \\
\mu_{1,1}+\cdots+\mu_{1,p}=1 \\
\mu_{1,1}+\cdots+\mu_{1,p}>1 \\
\end{array}
\right.,
\]
where $\mu_{1,1}+\cdots+\mu_{1,p}=E(\zeta_1)+\dots+E(\zeta_p)$.

\begin{assumption}\label{assump2}
We introduce the analogous of Assumption \ref{assump1} for the $\GINAR$ process.
\begin{enumerate}
\item\label{assump2:1}
The process is subcritical.

\item\label{assump2:2} 
The initial values $Z_{-p+1},\dots,Z_0$ and the variables in (\ref{def2}) all have finite second moments.

\item\label{assump2:3}
There is innovation that is $E(\eta)>0$.
\end{enumerate}
\end{assumption}

We can easily verify that in case of the $\GINAR$ process Assumption \ref{assump2} implies Assumption \ref{assump1}. By Proposition 2.1 of \cite{barczy} the condition $\varrho(\bm)<1$ is equivalent to $\mu_{1,1}+\cdots+\mu_{1,p}<1$ meaning that \ref{assump2:1} of Assumption \ref{assump2} and \ref{assump1:1} of Assumption \ref{assump1} are equivalent. 
It is obvious that \ref{assump2:2} of Assumption \ref{assump2} results \ref{assump1:3} of Assumption \ref{assump1}. 
The components of the vectors of (\ref{def}) are independent as only the first one is non-degenerate so \ref{assump1:2} of Assumption \ref{assump1} holds. The validity of \ref{assump1:4} of Assumption \ref{assump1} follows if the process does not die out that is guaranteed as by \ref{assump2:3} of Assumption \ref{assump2} there is innovation. The last assumption follows by the form of $\bmu$ and \ref{assump2:3} of Assumption \ref{assump2}.

In the simulation study we apply the function $\psi_{T}^{(3)}$ introduced in Subsection \ref{alternat} for the $\GINAR$ processes. 
As only the first type of the corresponding Galton--Watson process is not deterministic respect to the past  --- $R=\{1\}$ --- then by Remark \ref{reduced} and Subsection \ref{alternat} for any $T\in[0,\infty]$ we have that
\begin{equation}\label{inar1}
\psi_{T}^{(3)}\big((\widehat{\bI}_m^\cls|_R)^{-1/2}\widehat{\cY}_m^\cls|_R\big)=
\sup_{1\le k \le Tm} \frac{|a_1 \sum_{n=m+1}^{m+k} \widehat{M}_{m,n,1}^\cls|}{g_\gamma(m,k)}\DA \sup_{0\le t\le T} \frac{\left|W_1\left(\frac{t}{1+t}\right)\right|}{\left(\frac{t}{1+t}\right)^\gamma}
\end{equation}
as $m\to \infty$
where $a_1=a_1(m)$ is the $-1/2$-th power of the first element of $\widehat{\bI}_m^\cls$.
Similarly, if $\widetilde{\bJ}_1$ is invertible then we get the convergence
\begin{equation}\label{inar2}
\begin{split}
&\sup_{1\le k \le Tm} \frac{1}{g_\gamma(m,k)} \max \left\{ \left|\bb_1 \sum_{n=m+1}^{m+k} 
\left[
\begin{array}{c}
\widehat{M}_{m,n,1}^\cls\\
\widehat{N}_{m,n,1}^\cls
\end{array}
\right]
\right|,
\left|\bb_2 \sum_{n=m+1}^{m+k} 
\left[
\begin{array}{c}
\widehat{M}_{m,n,1}^\cls\\
\widehat{N}_{m,n,1}^\cls
\end{array}
\right]
\right|
\right\}\\
&
\DA \sup_{0\le t\le T} \max\left\{ \frac{\left|W_1\left(\frac{t}{1+t}\right)\right|}{\left(\frac{t}{1+t}\right)^\gamma}, \frac{\left|W_2\left(\frac{t}{1+t}\right)\right|}{\left(\frac{t}{1+t}\right)^\gamma}  \right\}, \qquad 0\le T \le \infty, \quad m\to \infty,
\end{split}
\end{equation}
where $\bb_1=\bb_1(m)$ and $\bb_2=\bb_2(m)$ are the rows of $(\widehat{\bJ}_{m,1}^\cls)^{-1/2}$ and $W_1(t), W_2(t), t\ge 0$ are independent one-dimensional standard Wiener processes. (Recall that $\widehat{\bJ}_{m,1}^\cls$ is the $\cls$ estimator of the first block of the block-diagonal matrix $\widetilde{\bJ}$.)

\newpage
\section{Simulation study}\label{simulation}

The procedures to detect model changes are based on the convergences in distribution shown in Theorem \ref{fotetel1}, and the consequences stated in Remark \ref{remark1}, Remark \ref{fotetelalt}, and Remark \ref{reduced}. The concrete test statistics that we are going to apply are introduced in Subsection \ref{alternat}. The related testing procedures are determined by these test statistics. The corresponding critical values are derived from the simulated critical values in Table 1 of \cite{horvath} worked out for testing procedures detecting changes in their linear models. The computation of these derived critical values has been discussed in Subsection \ref{alternat}.

For simplicity the tuning parameter $\gamma$ is set to $0.25$ throughout this section.

\subsection{2-type Galton--Watson process}
We test for a change in a 2-type
Galton--Watson process where we fix that the innovations $\eta_1$, and $\eta_2$ have Poisson($1$) distribution and the distribution of the number of offsprings of the same type, $\xi_{1,1},\xi_{2,2}$, is Bernoulli($.5$). These distributions are fixed in order to focus the simulation on the two types' impact on each other. 
We consider the cases when $T=1$ and $T=5$ where the test is based on the sample $\XX_0,\dots, \XX_{m+\lfloor m T \rfloor}.$
The number of repetitions are $1000$ for every parameter setup.
We apply the tests based on the convergence 
\[
\psi^{(3)}_{T}(\big(\widehat{\bI}_m^\cls\big)^{-1/2}\widehat{\cY}_m^\cls)
\DA
\left(\frac{T}{1+T}\right)^{1/2-\gamma} \sup_{0\le t \le 1} \max_{1\le i\le 2}\frac{|W_i(t)|}{t^\gamma}, \qquad 0\le T\le \infty, \qquad m \to \infty,
\]
with $T=\infty$ for the open-end and $T<\infty$ for the closed-end procedure.
In order to set the significance level of the test to $.05$ the one of the componentwise tests should be $1-\sqrt{1-.05}\sim .02532$ that unfortunately does not appear in the Table of \cite{horvath}. Instead, we use the value $.025$ that does appear in the table so the exact significance level of the test we perform is $\alpha=.049375$ in this subsection.

The next tables show the percentages of rejection under $\cH_0$ where $m=500$ and the number of offsprings of the opposite type, $\xi_{1,2}$ and $\xi_{2,1}$ are identically distributed Bernoulli($p$) with various $p\in [0,1]$ values.
\medskip

\hspace{-15pt}
 \begin{tabular}{|c| | c | c | c | c |}
\hline
$T=1$&  CLS  & CLS  & WCLS & WCLS \\ 
&open & closed & open & closed  \\ \hline \hline
$p=0$ & $1.8$ &  $5.9$ & $1.3$ & $5.9$ \\ \hline
$p=0.2$ & $2.2$  &  $8.9$ & $1.5$ & $7.6$ \\ \hline
$p=0.4$&  $3.4$  &  $10.0$ & $2.5$ & $6.8$ \\ \hline
    \hline
  \end{tabular} 
  \begin{tabular}{|c || c | c | c | c | c| }
\hline
$T=5$ &CLS  & CLS  & WCLS & WCLS \\ 
&open & closed & open & closed  \\ \hline \hline
$p=0$&$5.1$ & $7.5$ & $5.6$& $7.1$ \\ \hline
$p=0.2$ &$4.6$  &  $6.1$ & $5.0$ & $7.4$ \\ \hline
$p=0.4$ &$8.5$  &  $11.3$ & $6.1$ & $8.2$ \\ \hline
    \hline
  \end{tabular}

\medskip

\cite{horvath} and \cite{kirch} suggested the application of the open-end procedure even when the number of observations is limited as it has good results in case $T$ is big. Although, we present the proper closed-end procedures for any $T>0$. Our tables show that for $T=5$ --- meaning that $T$ is pretty big--- it is true that the open-end and closed-end procedures behave almost similarly as the fraction $T/(1+T)$ is near to $1$. On the contrary, for $T=1$ it is obvious that the open-end procedure's rejection rate is low, the closed-end one's rates approach $\alpha$ more accurately. This shows that the definition of the closed-end procedures is necessary in order to perform change-point detection when $T$ is small.

Next, we examine the rejection rates when $\cH_0$ does not hold. Let us consider the simple alternative hypothesis introduced in Subsection \ref{alternat} that the model is unchanged until the $(m+k^*)$-th step when a change occurs and later on there is no other model change.  The dynamics are as described above with the distributions of $\xi_{1,1}, \xi_{2,2}$ and $\eta_1, \eta_2$ fixed, and $\xi_{1,2}, \xi_{2,1}$ distributed Bernoulli($p_1$) before the change and Bernoulli($p_2$) after the change with $p_1,\, p_2 \in [0,1]$.
The rates are the following with fixed parameters $m=500, \, k^*=500$ and $T=2$  for the closed-end procedure. 
We show the results of the closed-end procedure as we have already seen that it is more effective when $T$ is small.
The critical values are the same as in the latter case with the exact significance level  $\alpha=.049375$.
\medskip

\hspace{-15pt}
 \begin{tabular}{|c| | c | c | c |}
\hline
CLS&  $p_2=0$  & $p_2=0.2$  & $p_2=0.4$   \\ \hline \hline
$p_1=0$ & 7.0 &  67.4 & 99.6 \\ \hline
$p_1=0.2$ &81.7 &  7.4 & 96.6  \\ \hline
$p_1=0.4$ & 100 &  97.8 & 10.7  \\ \hline
    \hline
  \end{tabular} \hspace{10pt}
  \begin{tabular}{|c| | c | c | c | }
\hline
WCLS&  $p_2=0$  & $p_2=0.2$  & $p_2=0.4$   \\ \hline \hline
$p_1=0$ & 5.8 &  50.1 & 97.7 \\ \hline
$p_1=0.2$ & 83.4 &  6.3 & 89.8 \\ \hline
$p_1=0.4$&  100  & 99.6  & 9.0  \\ \hline
    \hline
  \end{tabular}
\medskip

The diagonal contains the rejection rates for the models with no change, therefore the values are around $5\%$.  The off-diagonal elements --- where $\cH_0$ does not hold --- increase as the difference between the expected values before and after the change does. For example when it changes from  $p_1=0.4$ to $p_2=0$ then we reject in $100\%$ of the repetitions. 
\subsection{$\GINAR$ process}

As a special case the procedures are applicable to the $\GINAR$ processes. We show that the $\cls$ test based on the convergence in (\ref{inar2})
--- Type 2 --- have an advantage compare to the one based on (\ref{inar1}) --- Type 1. Namely, that it is more sensitive to changes not affecting the first moments of the distributions. The critical values are $\alpha=.05$ and $\alpha=.049375$ for the Type 1 and Type 2 tests, respectively, as in the first case the limit distribution is the function of a $1$-dimensional and in the second case a $2$-dimensional Wiener process. The second significance level follows as before.
Let us fix $m=100,$ $T=2,$ $k^*=100$ and let the innovation distribution be Poisson($1$). As $T$ is small we show the rejection rates related to the closed-end procedures. We suppose that there is exactly one change in the distribution of the number of offsprings from Bernoulli distributions to the Poisson ones as seen in the following table.

\medskip
\hspace{-20pt}
 \begin{tabular}{|c || c | c | c | c |}
\hline
Type 1 & Poi(0.2) & Poi(0.5)  &Poi(0.8)  \\ \hline \hline
Bern(0.2) & 7.0 & 54.0 & 99.8  \\ \hline
Bern(0.5) & 17.0  & 9.7 & 91.2  \\ \hline
Bern(0.8)&  66.5  &  43.7 & 19.0  \\ \hline
    \hline
  \end{tabular} \
 \begin{tabular}{|c || c | c | c | c |}
\hline
Type 2 & Poi(0.2) & Poi(0.5)  &Poi(0.8)   \\ \hline \hline
Bern(0.2) & 19.2 & 65.0 & 100.0  \\ \hline
Bern(0.5) & 24.8  &  29.5 &  96.9 \\ \hline
Bern(0.8)&  76.3  &  68.1 &  89.6 \\ \hline
    \hline
  \end{tabular}
\medskip

One can conclude that if the change does not occur in the expected values of the distribution --- as in the diagonal --- then the Type 1 test behaves nearly as under $\cH_0$ although the Type 2 test has higher rejection rates in the diagonal. Let us note that the rejection rate increases as the difference in the variances does. The variance of a Bernoulli($p$) distribution is $p(1-p)$ and of a Poisson($p$) distribution it is $p$. This means that for $p=.2$ the variance changes from $.16$ to $.2$ that results a modest rejection rate of $19.2\%$. Although, if $p=.8$ then the variances are $.16$ and $.8$ causing a higher rejection rate of $89.6\%$.
Let us recall that one of the conditions of Theorem \ref{alt} --- the theorem stating the strong consistency of the tests --- is the change of the expected values. Based on the simulation this condition seems to be unavoidable. 

\newpage
\section{Theoretical details and proofs}\label{proof}

\subsection{Moments and martingale differences}\label{mart}
In this subsection we examine the properties of the martingale differences introduced in Subsection \ref{intro}.
Let us recall the definitions
\[
\MM_n=\XX_n-E(\XX_n \mid \XX_{n-1}), \qquad \NN_n=\MM_n^2-E(\MM_n^2 \mid \XX_{n-1}),\qquad n=1,2,\dots
\]
For every $n=1,2,\dots$ and $i=1,\dots,p$ the conditional expected value of the $i$-th component of $\XX_n$ is
\begin{equation}\label{momentumok1}
E(X_{n,i}\mid \XX_{n-1})
=
E\left[
\sum_{k=1}^{X_{n-1,1}} \xi_{i,1}(n,k) + \cdots + \sum_{k=1}^{X_{n-1,p}} \xi_{i,p}(n,k) + \eta_i(n)
\right]
=
\bmu_i^\top\YY_{n-1}.
\end{equation}
Similarly, the conditional expected value of the $i$-th element of the vector $\MM_n^2$ is
\begin{equation}\label{momentumok2}
\begin{split}
&E(M_{n,i}^2\mid \XX_{n-1})=
E
\left[
\sum_{k=1}^{X_{n-1,1}}\!\! (\xi_{i,1}(n,k)-\mu_{i,1}) +\! \cdots \! +\! \!\!\sum_{k=1}^{X_{n-1,p}}\!\! (\xi_{i,p}(n,k)-\mu_{i,p}) \!+\! (\eta_i(n)-\mu_{i,\eta})
\right]^{2}
\\
&=
E\Bigg[
\sum_{k=1}^{X_{n-1,1}} \!\!(\xi_{i,1}(n,k)-\mu_{i,1})^2 + \cdots + \sum_{k=1}^{X_{n-1,p}}\!\! (\xi_{i,p}(n,k)-\mu_{i,p})^2 + (\eta_i(n)-\mu_{i,\eta})^2
\Bigg]
=
\bv_{i}^\top\YY_{n-1}
\end{split}
\end{equation}
by the independence of the random variables.
This means that $\MM_n=\XX_n-\bmu \YY_{n-1}$ and $\NN_n=\MM_n^2-\bV \YY_{n-1}$ for any $n=1,2,\dots$
The process satisfies the following proposition.
\begin{proposition}\label{prop1}
For any $\gamma\ge 1$ and $n\in \N$ the following statements hold:
\begin{enumerate}
\item\label{prop1:1}
\[
E\Big[\|\XX_n\|^\gamma \, \big|\, \XX_{n-1}\Big]
\le p^{{  \gamma}}(p+1)^\gamma\left\|
\YY_{n-1}
\right\| ^\gamma M_\gamma,
\]
\item\label{prop1:2}
\[
E\Big[\|\MM_n\|^\gamma \,\big|\, \XX_{n-1}\Big]
\le 
p^{{  \gamma}}(p+1)^\gamma\left\|
\YY_{n-1}
\right\|^\gamma  C_\gamma,
\]
\item\label{prop1:3}
\[
E\Big[\|\NN_n\|^\gamma \, \big|\, \XX_{n-1}\Big]
\le
2^{\gamma+1}
p^{{  3\gamma}}(p+1)^{2\gamma}\left\|
\YY_{n-1}
\right\|^{2\gamma}  C_{2\gamma},
\]
\end{enumerate}
where  $M_\gamma:=\max\limits_{1\le i,j\le p}\left\{E|\xi_{i,j}|^\gamma,E|\eta_i|^\gamma\right\}$, and  $C_\gamma:=\max\limits_{1\le i,j\le p}\left\{E|\xi_{i,j}-\mu_{i,j}|^\gamma,E|\eta_i-\mu_{i,\eta}|^\gamma\right\}$.
\end{proposition}
\begin{proof}
\ref{prop1:1}
For any $n\in \N$ and arbitrary $\bx=[x_1,\dots,x_p]^\top \in \R_{+}^p$ applying the Minkowski-inequality we get that
\begin{equation*}
\begin{split}
&\left(E\Big[\|\XX_n\|^\gamma \, \big|\, \XX_{n-1}=\bx\Big]\right)^{1/\gamma}=\left(
E\left[\left\|
 \sum_{i=1}^p \sum_{k=1}^{x_i} 
\bxi_{i}(n,k)
+\bbeta(n)
\right\|^\gamma\right]
\right)^{1/\gamma}
\\
&
\le
\left(
E\left[
\left( 
\sum_{j=1}^p 
\left|
\sum_{i=1}^p \sum_{k=1}^{x_i} \xi_{j,i}(n,k)+\eta_j(n)
\right|
\right)^\gamma
\right]
\right)
^{1/\gamma}
\\
&\le
\sum_{j=1}^p 
\left(
 \sum_{i=1}^p \sum_{k=1}^{x_i}\Big[E\left|\xi_{j,i}(n,k)\right|^\gamma\Big]^{1/\gamma}+\Big[E\left|\eta_j(n)\right|^\gamma\Big]^{1/\gamma}\right)
\\&
\le 
\sum_{j=1}^p 
(x_1+\dots+x_p+1)M_\gamma^{1/\gamma}
\le
p(p+1)
\left\|
\left[
\begin{array}{c}
\bx\\
1
\end{array}
\right]
\right\| M_\gamma^{1/\gamma}.
\end{split}
\end{equation*}
In the last step we used that 
\[
(x_1+\dots+x_p+1)\le (p+1)
\max\{x_1,\dots, x_p, 1\} \le
(p+1)
\left\|
\left[
\begin{array}{c}
\bx\\
1
\end{array}
\right]
\right\|, \qquad i=1,\dots,p.
\]
By summing up for all possible $x$ the proof is complete.

\ref{prop1:2}
The proof of \ref{prop1:2} is analogous to the previous one after the following step where all the notations are inherited from the proof of \ref{prop1:1}. We have that
\[
E\Big[
\|\MM_n\|^\gamma \,\big|\, \XX_{n-1}=\bx\Big]\le
E\left[
\left(
\sum_{j=1}^p \left| \sum_{i=1}^p \sum_{k=1}^{x_i} (\xi_{j,i}(n,k)-\mu_{j,i})+(\eta_j(n)-\mu_{j,\eta})\right|\right)^\gamma\right]
.
\]

\ref{prop1:3}
Let us note that for any vectors $\by=(y_1,\dots, y_p)\in \R^p$ and $\bz=(z_1, \dots, z_p)\in \R^p$ it holds that
\begin{equation*}
\begin{split}&
\|\by+\bz\|^\gamma
\le \left[
\sum_{i=1}^p |y_i+z_i|
\right]^\gamma
\le \left[
\sum_{i=1}^p \left(|y_i|+|z_i|\right)
\right]^\gamma
\le 2^\gamma p^\gamma \left[\max_{1\le i\le p}\{
|y_i|, |z_i|
\}\right]^\gamma\\
&
\le
2^\gamma p^\gamma \left[\max\{
\|\by\|, \|\bz\|
\}\right]^\gamma\le 2^\gamma p^\gamma \left[\|\by\|^\gamma+\|\bz\|^\gamma\right].
\end{split}
\end{equation*}
Therefore, applying the remarks and previous statements of the proof, and the Jensen inequality we get that 
\begin{equation*}
\begin{split}
&E\Big[\|\NN_n\|^\gamma \,\big|\, \XX_{n-1}\Big]=
E\left[\left\|
\MM_n^2-E[\MM_n^2 \mid \XX_{n-1}]
\right\|^\gamma \,\big|\, \XX_{n-1}\right]
 \\
&\le
2^\gamma p^\gamma \left(
E\left[\|\MM_n^2\|^\gamma \mid \XX_{n-1}\right]
+
\left\|
E\left[\MM_n^2 \mid \XX_{n-1}\right]
\right\|^\gamma
\right)
\le
2^{\gamma+1} p^\gamma  E\left[\|\MM_n\|^{2\gamma} \mid \XX_{n-1}\right] \\
&
\le
2^{\gamma+1} p^\gamma p^{2\gamma} (p+1)^{2\gamma}
\left\|
\YY_{n-1}
\right\|^{2\gamma}  C_{2\gamma}
=
2^{\gamma+1} p^{3\gamma} (p+1)^{2\gamma}
\left\|
\YY_{n-1}
\right\|^{2\gamma}  C_{2\gamma},
\end{split}
\end{equation*}
that completes the proof.
\end{proof}
In the following proposition we compute the elements of the covariance matrices of the martingale differences. The proof of \ref{fugg1} has already been given in (\ref{momentumok2}).
\begin{proposition}\label{fugg}
The latter statements hold for any $n=1,2,\dots$ and $i,j=1,\dots,p$, $i\neq j$, if \ref{assump1:1}-\ref{assump1:2} of Assumption \ref{assump1} are satisfied. 
\begin{enumerate}
\item\label{fugg1}
We have that
$E(M_{n,i}^2  | \XX_{n-1})=\bv_i^\top\YY_{n-1}$ resulting  $E(M_{n,i}^2)=\bv_i^\top E(\YY_{n-1})$. Furthermore, $E(M_{n,i}M_{n,j} | \XX_{n-1})=0$. 
\item\label{fugg2}
If the third moments of the variables in (\ref{def}) exist then $E(M_{n,i}N_{n,i} | \XX_{n-1})=\balpha_i^\top \YY_{n-1}$ and 
$E(M_{n,i}N_{n,i})=\balpha_i^\top E(\YY_{n-1})$. Also, $E(M_{n,i}N_{n,j} | \XX_{n-1})=0$.
 \item\label{fugg3}
If the fourth moments of the variables in (\ref{def}) exist then 
\[
E(N_{n,i}^2 \mid \XX_{n-1})=(\Beta_i-3\bv_i^2)^\top \YY_{n-1}+2\bv_i^\top\YY_{n-1}\YY_{n-1}^\top\bv_i
\]
and as a consequence
\[
E(N_{n,i}^2)=(\Beta_i-3\bv_i^2)^\top E(\YY_{n-1})+2\bv_i^\top E(\YY_{n-1}\YY_{n-1}^\top)\bv_i.
\]
Also, $E(N_{n,i}N_{n,j} | \XX_{n-1})=0$. As $E(M_{n,i}^4 | \XX_{n-1})=E(N_ {n,i}^2 | \XX_{n-1})+E(M_{n,i}^2 | \XX_{n-1})^2$ holds \ref{fugg1} implies that 
\[
E(M_{n,i}^4 \mid \XX_{n-1})=(\Beta_i-3\bv_i^2)^\top \YY_{n-1}+3\bv_i^\top \YY_{n-1}\YY_{n-1}^\top\bv_i.
\]
\end{enumerate}
\end{proposition}
\begin{proof}
\ref{fugg2}
By the definitions and simple calculations
\begin{equation*}
\begin{split}
&E[M_{n,i}N_{n,j} \mid \XX_{n-1}]=
E\left[M_{n,i}\left(M_{n,j}^2-E\left[M_{n,j}^2 \mid \XX_{n-1} \right]\right) \mid \XX_{n-1}\right]\\
&
=
E\left[M_{n,i}M_{n,j}^2 \mid \XX_{n-1}\right]-E\left[M_{n,i}\mid \XX_{n-1}\right]E\left[M_{n,j}^2 \mid \XX_{n-1}\right]
=E\left[M_{n,i}M_{n,j}^2 \mid \XX_{n-1}\right]\\
&
=
E\Bigg[
\left(
\sum_{k=1}^{X_{n-1,1}} (\xi_{i,1}(n,k)-\mu_{i,1}) + \cdots + \sum_{k=1}^{X_{n-1,p}} (\xi_{i,p}(n,k)-\mu_{i,p}) + (\eta_i(n)-\mu_{i,\eta})
\right)\\
&
\quad
\times
\left(
\sum_{k=1}^{X_{n-1,1}} (\xi_{j,1}(n,k)-\mu_{j,1}) + \cdots + \sum_{k=1}^{X_{n-1,p}} (\xi_{j,p}(n,k)-\mu_{j,p}) + (\eta_j(n)-\mu_{j,\eta})
\right)^2
\Bigg].
\end{split}
\end{equation*}
By the independence of the vectors and the components of the vectors
\begin{equation*}\label{momentumok3}
\begin{split}
&
E(M_{n,i}N_{n,i} \mid \XX_{n-1})=E(M_{n,i}^3 \mid \XX_{n-1})\\
&=\sum_{k=1}^{X_{n-1,1}}(\xi_{i,1}(n,k)-\mu_{i,1})^3 + \cdots + \sum_{k=1}^{X_{n-1,p}}(\xi_{i,p}(n,k)-\mu_{i,p}) ^3+ (\eta_i(n)-\mu_{i,\eta})^3
=\balpha_i^\top\YY_{n-1},
\end{split}
\end{equation*}
and $E(M_{n,i}N_{n,j} | \XX_{n-1})=0$. 

\ref{fugg3}
Applying the definition of the martingale differences we have that
\begin{equation*}
\begin{split}
&E[N_{n,i}N_{n,j} \mid \XX_{n-1}]=
E\left[\left(M_{n,i}^2-E\left[M_{n,i}^2 \mid \XX_{n-1}\right]\right)\left(M_{n,j}^2-E\left[M_{n,j}^2 \mid \XX_{n-1} \right]\right) \mid \XX_{n-1}\right]\\
&
=
E\left[M_{n,i}^2M_{n,j}^2 \mid \XX_{n-1}\right]-E\left[M_{n,i}^2 \mid \XX_{n-1}\right]E\left[M_{n,j}^2 \mid \XX_{n-1}\right]\\
\end{split}
\end{equation*}
where $E[M_{n,i}^2|\XX_{n-1}]E[M_{n,j}^2|\XX_{n-1}]=\bv_i^\top\YY_{n-1}\YY_{n-1}^\top \bv_j $ by \ref{fugg1} and
\begin{equation*}
\begin{split}
E
&\left[M_{n,i}^2M_{n,j}^2 \mid \XX_{n-1}\right]\\
&
=
E\Bigg[
\left(
\sum_{k=1}^{X_{n-1,1}} (\xi_{i,1}(n,k)-\mu_{i,1}) + \cdots + \sum_{k=1}^{X_{n-1,p}} (\xi_{i,p}(n,k)-\mu_{i,p}) + (\eta_i(n)-\mu_{i,\eta})
\right)^2\\
&
\quad
\times
\left(
\sum_{k=1}^{X_{n-1,1}} (\xi_{j,1}(n,k)-\mu_{j,1}) + \cdots + \sum_{k=1}^{X_{n-1,p}} (\xi_{j,p}(n,k)-\mu_{j,p}) + (\eta_j(n)-\mu_{j,\eta})
\right)^2
\Bigg].
\end{split}
\end{equation*}
By the independence of the variables the products with a term on the first power (and therefore with one on the third power) have $0$ expected value. First, we assume that $i=j$. Then the expected value of the sum of the fourth powers is $\Beta_i^\top \YY_{n-1}$. Next, we consider the cases with two squared terms. The sum of the expected values of the ones with two different offsprings of the same type and two squared terms of different types is
\begin{equation*}
\begin{split}
\sum_{k=1}^p& \binom{4}{2}\binom{X_{n-1,k}}{2}v_{i,k}^2 + \binom{4}{2}\sum_{\substack{k,\ell=1\\k < \ell}}^p X_{n-1,k}X_{n-1,\ell} v_{i,k}v_{i,\ell}+\binom{4}{2}\sum_{k=1}^p X_{n-1,k}v_{i,k}v_{i,\eta}\\
&
=3\sum_{k,l=1}^pX_{n-1,k}X_{n-1,\ell} v_{i,k} v_{i,\ell} -3\sum_{k=1}^pX_{n-1,k}v_{i,k}^2+6\sum_{k=1}^p X_{n-1,k} v_{i,k} v_{i,\eta}\\
&=\left[
3\bv_i^\top\YY_{n-1}\YY_{n-1}^\top \bv_i-6\sum_{k=1}^p X_{n-1,k}v_{i,k}v_{i,\eta}-3v_{i,\eta}^2\right]
-
\left[3\left(\bv_i^2\right)^\top \YY_{n-1}-3v_{i,\eta}^2\right]\\ 
&+6\sum_{k=1}^p X_{n-1,k} v_{i,k} v_{i,\eta}
=3\bv_i^\top\YY_{n-1}\YY_{n-1}^\top \bv_i-3\left(
\bv_i^2\right)
^\top
\YY_{n-1}.
\end{split}
\end{equation*}
Then $E(M_{n,i}^4 |\XX_{n-1})=\Beta_{i}^\top\YY_{n-1}+3\left(\bv_i^\top\YY_{n-1}\right)^2-3(\bv_i^2)^\top\YY_{n-1}$
 that implies 
\begin{equation*}\label{momentumok4}
E(N_{n,i}^2 \mid \XX_{n-1})=(\Beta_i-3\bv_i^2)^\top\YY_{n-1}+
2\left(\bv_i^\top\YY_{n-1}\right)^2, \qquad i=1,\dots,p,
\end{equation*}
and 
$
E(M_{n,i}^2M_{n,j}^2 | \XX_{n-1})=\bv_i^\top\YY_{n-1}\YY_{n-1}^\top\bv_j
$ resulting $E(N_{n,i}^2N_{n,j}^2 | \XX_{n-1})=0$.
\end{proof}

\subsection{Parameter estimations}\label{estim}

We define the $\cls$ (Conditional Least Squares) estimators of the parameters $\bmu, \, \bV, \, \bA$ and $\bB$ motivated by the method of \citet{klimko} worked out for linear models.
To get the estimator of $\bmu$ we minimize the sum
\[
Q_m^2=
\frac{1}{2} \sum_{n=1}^m
\left[
\XX_n-E(\XX_n \mid \XX_{n-1})
\right]^\top
\left[
\XX_n-E(\XX_n \mid \XX_{n-1})
\right]
=
\frac{1}{2} \sum_{n=1}^m \sum_{i=1}^p (X_{n,i}-\bmu_i^\top\YY_{n-1})^2
\]
by taking the derivative of the expression with respect to the rows of $\bmu=[\bmu_1,\ldots,\bmu_p]^\top$
so the following equation system has to be solved:
\begin{equation}\label{eqsys}
\sum_{n=1}^m \bigg[X_{n,i}-\bmu_i^\top\YY_{n-1}\bigg]  \YY_{n-1}^\top =0, \qquad i=1,\dots, p.
\end{equation}
We applied formula (\ref{momentumok1}) for the conditional expected value $E(\XX_n | \XX_{n-1})$.
In shorter form the equation system can be written as
\[
0=\nabla_{\bmu}Q_m^2
=
\sum_{n=1}^m 
\left(
\XX_n-\bmu\YY_{n-1}
\right)
\YY_{n-1}^\top=
\sum_{n=1}^m \MM_n
\YY_{n-1}^\top.
\]
Solving for $\bmu$ we get that the $\cls$ estimator of $\bmu$ based on $\XX_0,\ldots,\XX_m$ is
\[
\widehat{\bmu}_m^\cls=\
\left[
\sum_{n=1}^m
\XX_n
\YY_{n-1}^\top
\right]
\left[
\sum_{n=1}^m \YY_{n-1}\YY_{n-1}^\top
\right]^{-1}.
\]
Similarly, we define the estimator of $\bV$ as the matrix that minimizes
\[
\frac{1}{2} \sum_{n=1}^m [\MM_n^2-E(\MM_n^2 \mid \XX_{n-1})]^\top  [\MM_n^2-E(\MM_n^2 \mid \XX_{n-1})]\!=\!\!\frac{1}{2} \sum_{n=1}^m \! \sum_{i=1}^p \! \big[(X_{n,i}-\bmu_i^\top\YY_{n-1})^2-\bv_i^\top \YY_{n-1}\big]^2
\]
where we applied (\ref{momentumok2}) to extract the conditional expected value.
We replace the vectors $\bmu_i$ by the already defined estimators $\widehat{\bmu}_{m,i}^\cls$, $i=1,\dots,p$. Therefore we minimize
\[
\frac{1}{2} \sum_{n=1}^m\sum_{i=1}^p \left[
\left(
X_{n,i}-(\widehat{\bmu}_{m,i}^\cls)^\top\YY_{n-1}
\right)^2-
\bv_i^\top\YY_{n-1}
\right]^2=
\frac{1}{2} \sum_{n=1}^m\sum_{i=1}^p \left[
(\widehat{M}_{m,n,i}^\cls)^2-
\bv_i^\top\YY_{n-1}
\right]^2
\]
where $(\widehat{\bmu}_{m,i}^\cls)^\top$ and $\widehat{M}_{m,n,i}^\cls$ denote the $i$-th row of $\widehat{\bmu}_m^\cls$ and $\widehat{\MM}_{m,n}^\cls=\XX_n-\widehat{\bmu}_m ^\cls \YY_{n-1}$, respectively.
Applying the previously seen method after differentiation and solving for $\widehat{\bV}_m^\cls$ we get that the $\cls$ estimator of $\bV$ based on the training sample is
\begin{equation}\label{k}
\widehat{\bV}_m^\cls=
\left[
\sum_{n=1}^m
\big(\widehat{\MM}_{m,n}^\cls\big)^{2}\YY_{n-1}^\top 
\right]
\left[
\sum_{n=1}^m \YY_{n-1}\YY_{n-1}^\top 
\right]^{-1}.
\end{equation}
The formula for $\widehat{\bA}_{m}^\cls$ follows similarly if we minimize
\[
\frac{1}{2}  \sum_{n=1}^m [\MM_n^3-E(\MM_n^3 \mid \XX_{n-1})]^\top \! [\MM_n^3-E(\MM_n^3 \mid \XX_{n-1})]\!=\!\frac{1}{2} \sum_{n=1}^m \! \sum_{i=1}^p \!
\big[(X_{n,i}-\bmu_i^\top\YY_{n-1})^3-\balpha_{i}^\top \YY_{n-1}\big]^2\!\!.
\]
By replacing $\bmu_i$ with $\widehat{\bmu}_{m,i}^\cls$ and solving for $\widehat{\bA}_{m}^\cls$ we get that
\begin{equation*}\label{kk}
\widehat{\bA}_{m}^\cls=
\left[
\sum_{n=1}^m
\big(\widehat{\MM}_{m,n}^\cls\big)^{3}\YY_{n-1}^\top 
\right]
\left[
\sum_{n=1}^m \YY_{n-1}\YY_{n-1}^\top 
\right]^{-1}.
\end{equation*}
Finally, to determine the $\cls$ estimator of  $\bB$ we minimize the sum
\begin{equation*}
\begin{split}
&
\frac{1}{2}\sum_{n=1}^m [\MM_n^4-E(\MM_n^4 \mid \XX_{n-1})]^\top[\MM_n^4-E(\MM_n^4 \mid \XX_{n-1})]
=
\frac{1}{2}\sum_{n=1}^m\sum_{i=1}^p \left[
M_{n,i}^4-E[M_{n,i}^4 \mid \XX_{n-1}]
\right]^2\\
&=
\frac{1}{2}\sum_{n=1}^m\sum_{i=1}^p \left[
M_{n,i}^4-3(\bv_i^\top\YY_{n-1})^2-(\Beta_{i}-3\bv_i^2)^\top \YY_{n-1}
\right]^2
\end{split}
\end{equation*}
by \ref{fugg3} of Proposition \ref{fugg}.
By replacing the already estimated terms with the corresponding estimators and solving for $\bB$ we get that
\begin{equation*}
\widehat{\bB}_{m}^\cls=\left[
\sum_{n=1}^m\widehat{\KK}_{m,n}^\cls\YY_{n-1}^\top
\right]
\left[
\sum_{n=1}^m \YY_{n-1}\YY_{n-1}^\top ,
\right]^{-1}
\end{equation*}
with
\[
\widehat{\KK}_{m,n}^\cls=(\widehat{\MM}_{m,n}^\cls)^4
-3(\widehat{\bV}_m^\cls\YY_{n-1})^2
+3(\widehat{\bV}_m^\cls)^{(2)}\YY_{n-1}, \qquad n=1,2,\dots
\]
where $(\widehat{\bV}_{m}^\cls)^{(2)}$ is defined as $[(\widehat{\bv}_{m,1}^\cls)^{2},\dots, (\widehat{\bv}_{m,p}^\cls)^{2}]^\top$.
\begin{remark}\label{remark3}
In the equation system (\ref{eqsys}) the rows of $\bmu$ appear in distinct equations. Therefore the $\cls$ estimators of  $\bmu_1^\top, \dots, \bmu_p^\top$ can be computed independently as
\[
\left(\widehat{\bmu}_{m,i}^\cls\right)^\top=
\left[
\sum_{n=1}^m
X_{n,i}\YY_{n-1}^\top 
\right]
\left[
\sum_{n=1}^m \YY_{n-1}\YY_{n-1}^\top 
\right]^{-1}, \qquad i=1,\dots,p.
\]
Similarly, the rows of $\bV, \bA, \bB$ can also be estimated separately, namely for any $i=1,\dots,p$
\begin{equation}\label{sigma}
\begin{split}
&\left(\widehat{\bv}_{m,i}^\cls\right)^\top=
\left[
\sum_{n=1}^m
\Big(\widehat{M}_{m,n,i}^\cls\Big)^2
\YY_{n-1}^\top 
\right]
\left[
\sum_{n=1}^m \YY_{n-1}\YY_{n-1}^\top 
\right]^{-1},\\
&
\left(
\widehat{\balpha}_{m,i}^\cls\right)^\top=
\left[
\sum_{n=1}^m
\Big(\widehat{M}_{m,n,i}^\cls\Big)^3
\YY_{n-1}^\top 
\right]
\left[
\sum_{n=1}^m \YY_{n-1}\YY_{n-1}^\top 
\right]^{-1},
\\
&
\left(\widehat{\Beta}_{m,i}^\cls\right)^\top=
\left[
\sum_{n=1}^m
\widehat{K}_{m,n,i}^\cls
\YY_{n-1}^\top 
\right]
\left[
\sum_{n=1}^m \YY_{n-1}\YY_{n-1}^\top 
\right]^{-1},\\
\end{split}
\end{equation}
where $\widehat{K}_{m,n,i}^\cls $ is the $i$-th component of the previously defined vector $\widehat{\KK}_{m,n}^\cls$.
Therefore, if some rows of the matrices $\bmu, \, \bV, \, \bA$ and $\bB$ are a priori given then the rest of the rows can be estimated as seen here. For example if the process is $\GINAR$ then all we have to estimate are $\bmu_1^\top, \ \bv_1^\top, \ \balpha_{1}^\top$ and $\Beta_{1}^\top$ as the rest of the rows are known.
\end{remark}

We also define another type of parameter estimators called the Weighted Conditional Least Squares (WCLS) estimators. The weighted version of the CLS estimation was introduced by \citet{nelson} with a general weight function to estimate the parameters in multivariate linear regression models. The WCLS estimation used in our paper is a special case of Nelson's method and it is defined as the CLS estimation based on the weighted process $\XX_n':=\XX_n/\sqrt{\bone^\top  \YY_{n-1}}$, $n=1,2,\ldots$ Our definition is originated from \citet{wei} and \citet{winnicki} who used the WCLS estimation to estimate the mean and the variance of the offspring and the innovation distribution in single-type Galton--Watson processes.
We also consider the weighted versions of the sequences of martingale differences 
\[
\MM'_n:=\XX_n'-E(\XX_n' \mid \XX_{n-1})=\frac{\MM_n}{\sqrt{\bone^\top\YY_{n-1}}},  \qquad \NN'_n:=\MM'^2_n-E(\MM'^{2}_n \mid \XX_{n-1})=\frac{\NN_n}{\bone^\top\YY_{n-1}}
\]
for $n=1,2,\dots$, where we applied the formulas (\ref{momentumok1}) and (\ref{momentumok2}).
The estimators are given by 
\begin{equation}\label{sigmaw}
\begin{split}
&
\widehat{\bmu}_m^{\wcls}=
\left[
\sum_{n=1}^m
\frac{
\XX_n
\YY_{n-1}^\top }
{\bone^\top \YY_{n-1}}
\right]
\left[
\sum_{n=1}^m 
\frac{
\YY_{n-1}\YY_{n-1}^\top }
{\bone^\top \YY_{n-1}}
\right]^{-1},\\
&
\widehat{\bV}_m^\wcls=
\Bigg[
\sum_{n=1}^m
\frac{\left(
\widehat{\MM}_{m,n}^\wcls\right)^{2}\YY_{n-1}^\top 
}{
\bone^\top \YY_{n-1}}
\Bigg]
\left[
\sum_{n=1}^m 
\frac{
\YY_{n-1}\YY_{n-1}^\top }
{(\bone^\top \YY_{n-1})^2}
\right]^{-1},
\\
&
\widehat{\bA}_{m}^\wcls=
\Bigg[
\sum_{n=1}^m
\frac{\left(
\widehat{\MM}_{m,n}^\wcls\right)^{3}\YY_{n-1}^\top 
}{
(\bone^\top \YY_{n-1})^{3/2}}
\Bigg]
\left[
\sum_{n=1}^m
\frac{
\YY_{n-1}\YY_{n-1}^\top }
{(\bone^\top \YY_{n-1})^3}
\right]^{-1},\\
&
\widehat{\bB}_{m}^\wcls=\left[\sum_{n=1}^m 
\frac{\widehat{\KK}_{m,n}^\wcls\YY_{n-1}^\top}{(\bone^\top \YY_{n-1})^2}
\right]
\left[\sum_{n=1}^m 
\frac{\YY_{n-1}\YY_{n-1}^\top}{(\bone^\top \YY_{n-1})^4}
\right]^{-1},
\end{split}
\end{equation}
with $\widehat{\MM}_{m,n}^\wcls=\XX_n/(\sqrt{\bone^\top \YY_{n-1}})-\widehat{\bmu}_m^\wcls \YY_{n-1}/(\sqrt{\bone^\top \YY_{n-1}})$ and
\[
\widehat{\KK}_{m,n}^\wcls=(\widehat{\MM}_{m,n}^\wcls)
^4-3\frac{(\widehat{\bV}_m^\wcls\YY_{n-1})^2}{\left(\bone^\top  \YY_{n-1}\right)^2}+3\frac{(\widehat{\bV}^\wcls_m)^{(2)}\YY_{n-1}}{(\bone^\top \YY_{n-1})^2}, \qquad n=1,2,\dots
\]
where $(\widehat{\bV}_{m}^\wcls)^{(2)}$ is defined as $[(\widehat{\bv}_{m,1}^\wcls)^{2},\dots, (\widehat{\bv}_{m,p}^\wcls)^{2}]^\top$.
The following proposition gives sufficient conditions providing the existence of the parameter estimators.
\begin{proposition}\label{prop}

\begin{enumerate}
\item\label{prop:1}
If the process is stable and the variables in (\ref{def}) have finite second moments then $E(\widetilde{\YY}\widetilde{\YY}^\top)$ is non-degenerate exactly if \ref{assump1:4}-\ref{assump1:6} of Assumption \ref{assump1} hold.
\item\label{prop:2}
As a consequence the parameter estimators exist with probability tending to $1$ as $m\to \infty$.
\end{enumerate}
\end{proposition}

\begin{proof} \ref{prop:1}
Theorem 2 of \cite{szucs} states that the components of $\widetilde{\XX}$  are linearly independent if and only if \ref{assump1:4}-\ref{assump1:6} of Assumption \ref{assump1} hold.
Therefore, we only have to show that the positive semi-definite matrix
$E\big(\widetilde{\YY}\widetilde{\YY}^\top\big)$ is degenerate exactly if the components of $\widetilde{\XX}$  are linearly dependent, meaning that there is a vector  $\bc\in\R^p$, $\bc\neq\bnull$, satisfying
$
\bc^\top\big(\widetilde{\XX}-E(\widetilde{\XX})\big) = 0$
with probability $1$.
If the matrix is degenerate then there exists a vector
\[
\bd=\left[
\begin{array}{c}
\bc \\
c'
\end{array}
\right]
\in\R^p\times\R=\R^{p+1},
\qquad
\bd\neq\bnull,
\]
satisfying
$
0 = \bd^\top E(\widetilde{\YY}\widetilde{\YY}^\top)\bd = E ( \bd^\top\widetilde{\YY}\widetilde{\YY}^\top\bd)
= E ( \bd^\top\widetilde{\YY})^2.
$
This holds if and only if $\bc^\top\widetilde{\XX}+c'=\bd^\top\widetilde{\YY}=0$ almost surely meaning that if
 $E(\widetilde{\YY}\widetilde{\YY}^\top)$ is degenerate then $\bc^\top E(\widetilde{\XX})+c'=0$, so as a consequence
$
\bc^\top(\widetilde{\XX}-E(\widetilde{\XX})) = 0$.
Let us note that currently $\bc\neq\bnull$ as $\bc=\bnull$ results that $c'=0$ and $\bd=\bnull$. This means that the components of $\widetilde{\XX}$ are linearly dependent. 

Let us verify the other implication. If the components of $\widetilde{\XX}$ are linearly dependent with some vector  $\bc\neq\bnull$ then with $c'=-\bc^\top E(\widetilde{\XX})$ and $\bd\neq\bnull$ it holds that
\[
\bd^\top\widetilde{\YY}=\bc^\top\widetilde{\XX}+c'=\bc^\top\Big(\widetilde{\XX}-E(\widetilde{\XX})\Big) = 0 \qquad \textrm{a.s.}
\]
meaning that $E\big(\widetilde{\YY}\widetilde{\YY}^\top\big)$ is not positive definite.

\ref{prop:2}
By ergodicity and \ref{prop:1} the statement obviously holds for the $\cls$ estimators where the first terms divided by $m$ and the second terms multiplied with $m$ exist with probability tending to $1$. The latter is true as by ergodicity
\[
\frac{1}{m}\sum_{n=1}^m \YY_{n-1}\YY_{n-1}^\top \to E(\widetilde{\YY}\widetilde{\YY}^\top), \qquad m\to \infty.
\]
Next we show that for a sequence of non-negative random variables $S_{n}$, $n\in \N$, the limit of 
$
\frac{1}{m}\sum_{n=1}^m (\YY_{n-1}\YY_{n-1}^\top)/(1+S_n)
$
is invertible with probability $1$. (Let us note that this limit exists with probability $1$.)
Our aim is to show that the limit matrix is not only positive semi-definite but also positive definite that is for any vector $\bnull \neq \bv\in \R^{p+1}$ it holds that 
\[
\bv^\top \left(\lim_{m\to \infty} \frac{1}{m}\sum_{n=1}^m \frac{\YY_{n-1}\YY_{n-1}^\top}{1+S_n}\right)\bv >0.
\] By \ref{prop:1} we know that there is an index $n_\bv\in \N$ such that $\bv^\top \YY_{n_\bv-1}\YY_{n_\bv-1}^\top \bv>0$ and of course $\bv^\top \YY_{n-1}\YY_{n-1}^\top \bv\ge 0$ for every $n\in \N$. As the denominators $1+S_n$ are strictly positive for every $n\in \N$ we have that with the same index $n_\bv$ the following inequalities hold:
\[
\bv^\top \frac{\YY_{n_\bv-1}\YY_{n_\bv-1}^\top}{1+S_{n_\bv}} \bv>0, \qquad \bv^\top \frac{\YY_{n-1}\YY_{n-1}^\top}{1+S_n} \bv\ge 0, \qquad n\in \N.
\]
This completes the proof.
\end{proof}
In the following theorem we examine the asymptotic behaviors of the introduced parameter estimators.
\begin{theorem}\label{normal1} The following statements hold if the null hypothesis $\cH_0$ and Assumption \ref{assump1} are satisfied.
\begin{enumerate}
\item\label{normal1:1}
If for some $\varepsilon>0$ the $(4+\varepsilon)$-th, $(6+\varepsilon)$-th, $(2+\varepsilon)$-th, and fourth moments of the variables in (\ref{def}) exist then $\sqrt{m}(\widehat{\bmu}_{m,i}^\cls-\bmu_i),\, \sqrt{m}(\widehat{\bv}_{m,i}^\cls-\bv_i),\, \sqrt{m}(\widehat{\bmu}_{m,i}^\wcls-\bmu_i)$, and $\sqrt{m}(\widehat{\bv}_{m,i}^\wcls-\bv_i)$ are asymptotically
 normal, respectively, for any $i=1,\dots,p$ as $m\to \infty$. As a consequence $\sqrt{m} (\widehat{\bmu}_m^\cls-\bmu),\, \sqrt{m} (\widehat{\bV}_m^\cls-\bV),\, \sqrt{m} (\widehat{\bmu}_m^\wcls-\bmu)$, and $\sqrt{m} (\widehat{\bV}_m^\wcls-\bV)$ are $O_P(1)$, respectively.

\item\label{eltolas:1}
If the variables in (\ref{def}) have finite second, third, fourth, and fifth moments then the estimators $\widehat{\bmu}_m^\cls, \widehat{\bV}_m^\cls, \widehat{\bA}_m^\cls$, and $\widehat{\bB}_m^\cls$ are strongly consistent, respectively.
\item\label{eltolas:2}
Since Assumption \ref{assump1} results that the second moments of  the variables in (\ref{def}) are finite
the estimators $\widehat{\bmu}_m^\wcls, \widehat{\bV}_m^\wcls$ are strongly consistent.
If additionally the variables in (\ref{def}) have finite third moments then the estimators $ \widehat{\bA}_m^\wcls$ and $\widehat{\bB}_m^\wcls$ are also strongly consistent. 
\end{enumerate}
\end{theorem}

\begin{proof}
At several points of our proof we will apply the multidimensional Martingale Central Limit Theorem (MCLT). For reference see e.g. \citet{jacod}, Chapter VIII, Theorem 3.33.

\ref{normal1:1}
Applying Remark \ref{remark3} we get that
\begin{equation*}
\begin{split}
\sqrt{m}
\left[
\widehat{\bmu}_{m,i}^\cls-\bmu_i
\right]^\top
&=
\left[
\frac{1}{\sqrt{m}}
\sum_{n=1}^m
\left[
X_{n,i}-\bmu_i^\top\YY_{n-1}
\right]
\YY_{n-1}^\top 
\right]
\left[
\frac{1}{m}
\sum_{n=1}^m \YY_{n-1}\YY_{n-1}^\top \right]^{-1}
\\
&
=
\left[
\frac{1}{\sqrt{m}}
\sum_{n=1}^m
M_{n,i}\YY_{n-1}^\top 
\right]
\left[
\frac{1}{m}
\sum_{n=1}^m \YY_{n-1}\YY_{n-1}^\top \right]^{-1}.
\end{split}
\end{equation*}
By ergodicity as the second moments exist
\[
\frac{1}{m}\sum_{n=1}^m\YY_{n-1}\YY_{n-1}^\top \to E\left[\widetilde{\YY}\widetilde{\YY}^\top \right] \qquad \textrm{a.s.}, \qquad m\to \infty.
\]
Let us check that the conditions of the Martingale Central Limit Theorem are satisfied if we apply it to the sequence $\frac{1}{\sqrt{m}}\sum_{n=1}^m M_{n,i}\YY_{n-1}^\top , \ m=1,2,\dots$ First of all, 
\[
E\left[
\frac{1}{\sqrt{m}}\sum_{n=1}^m M_{n,i}\YY_{n-1}^\top \mid\XX_{n-1}
\right]
=
\frac{1}{\sqrt{m}}\sum_{n=1}^mE[M_{n,i}\mid \XX_{n-1}]\YY_{n-1}^\top=\bnull
\]
for any $m=1,2,\dots$ 
Let us check the Lindeberg condition:
\begin{equation*}
\begin{split}
&
\sum_{n=1}^m E
\left[
\left\|
\frac{1}{\sqrt{m}}M_{n,i} \YY_{n-1}^\top 
\right\|^2
\chi_{\big\{\left\|
\frac{1}{\sqrt{m}} M_{n,i}\YY_{n-1}^\top 
\right\|>\delta\big\}} \, \Big| \, \XX_{n-1}
\right]\\
&
\le
\frac{1}{\delta^\varepsilon}\sum_{n=1}^m E\left[
\left\|
\frac{1}{\sqrt{m}} M_{n,i}\YY_{n-1}^\top 
\right\|^{2+\varepsilon} \, \Big| \, \XX_{n-1}
\right]
\le
\frac{1}{\delta^\varepsilon m^{1+\varepsilon/2}}\sum_{n=1}^mE
\Big[
|M_{n,i}|^{2+\varepsilon}\|\YY_{n-1}\|^{2+\varepsilon} \mid \XX_{n-1}
\Big]\\
&
=
\frac{1}{\delta^\varepsilon m^{1+\varepsilon/2}}\sum_{n=1}^m
E\Big[|M_{n,i}|^{2+\varepsilon}\mid \XX_{n-1}\Big]\|\YY_{n-1}\|^{2+\varepsilon}\le \frac{C}{m^{1+\varepsilon/2}}\sum_{n=1}^m \left\| \YY_{n-1}\right\|^{4+2\varepsilon}\to 0, \quad m\to \infty,
\end{split}
\end{equation*}
where the last step holds by \ref{prop1:2} of Proposition \ref{prop1}.
Also, as the third moments exist then by ergodicity we have
\begin{equation*}
\begin{split}
&
\frac{1}{m} \sum_{n=1}^m  E
\left[
\Big(M_{n,i}\YY_{n-1}^\top \Big)\Big(M_{n,i}\YY_{n-1}^\top
\Big)^\top \,\Big|\,\XX_{n-1}
\right]
=
\frac{1}{m} \sum_{n=1}^m \YY_{n-1}^\top E(M_{n,i}^2 \mid \XX_{n-1})\YY_{n-1}
\\
&=
\frac{1}{m} \sum_{n=1}^m \YY_{n-1}^\top 
\left(
\bv_i^\top\YY_{n-1}
\right) 
\YY_{n-1}
\to
E\left[
\widetilde{\YY}^\top \left(
\bv_i^\top\widetilde{\YY}
\right) 
\widetilde{\YY}\right] \qquad \textrm{a.s.}, \qquad m\to \infty,
\end{split}
\end{equation*}
that enables us to determine the covariance matrix.
 So by the Central Limit Theorem
\[
\sqrt{m}\left[\widehat{\bmu}_{m,i}^\cls-\bmu_i\right]^\top \DA \cN(0,\Sigma), \qquad m\to \infty, \qquad i=1,\dots,p,
\]
where
\[
\Sigma=\left(
E\left[\widetilde{\YY}\widetilde{\YY}^\top\right]
\right)^{-1}E\left[\widetilde{\YY}\left(
\bv_i^\top\widetilde{\YY}
\right)
\widetilde{\YY}^\top \right]
\left(
E\left[\widetilde{\YY}\widetilde{\YY}^\top\right]
\right)^{-1}.
\]

Next, we show that $\sqrt{m}(\widehat{\bmu}_{m,i}^\wcls-\bmu_i)$ is asymptotically normal. By the formula of the estimator
\begin{equation*}
\begin{split}
\sqrt{m}\left[
\widehat{\bmu}_{m,i}^\wcls-\bmu_i\right]^\top&=
\sqrt{m}
\left[ \sum_{n=1}^m 
\frac{
X_{n,i}- \bmu_i^\top \YY_{n-1}}{\bone^\top \YY_{n-1}}\YY_{n-1}^{\top}
\right]
\left[ \sum_{n=1}^m \frac{\YY_{n-1}\YY_{n-1}^{\top}}{\bone^\top \YY_{n-1}}\right]^{-1}\\
&
=\left[\frac{1}{\sqrt{m }} \sum_{n=1}^m \frac{M_{n,i}\YY_{n-1}^{\top}}{\bone^\top \YY_{n-1}}\right]
\left[\frac{1}{m} \sum_{n=1}^m\frac{\YY_{n-1}\YY_{n-1}^{ T}}{\bone^\top \YY_{n-1}}\right]^{-1},
\end{split}
\end{equation*}
and as the second moments are finite we can apply ergodicity so
\[
\left[\frac{1}{m} \sum_{n=1}^m \frac{\YY_{n-1}\YY_{n-1}^\top }{\bone^\top \YY_{n-1}}\right]^{-1}\to \left( E\left[\frac{\widetilde{\YY}\widetilde{\YY}^\top}{1^\top \widetilde{\YY}}\right]\right)^{-1} \qquad
 \textrm{a.s.}, \qquad m\to \infty.
\]
Let us check that the sequence $\frac{1}{\sqrt{m}} \sum_{n=1}^mM_{n,i}\YY_{n-1}^{\top}/(\bone^\top \YY_{n-1}), \ m=1,2,\dots$,
satisfies the conditions of the Martingale Central Limit Theorem. 
First of all,
\[
\frac{1}{\sqrt{m }} \sum_{n=1}^mE\left[\frac{M_{n,i}\YY_{n-1}^{\top}}{\bone^\top \YY_{n-1}}\, \Big|\, \XX_{n-1}\right]
=\frac{1}{\sqrt{m }} \sum_{n=1}^mE\left[M_{n,i} \mid \XX_{n-1}\right]\frac{\YY_{n-1}^{\top}}{\bone^\top \YY_{n-1}}=\bnull,
\]
and
\begin{equation*}
\begin{split}
&\sum_{n=1}^m
E\left[\left\| \frac{1}{\sqrt{m}} \frac{M_{n,i}\YY_{n-1}^{\top}}{\bone^\top \YY_{n-1}} \right\|^2
\chi_{\left\| \frac{1}{\sqrt{m}} \frac{M_{n,i}\YY_{n-1}^{\top}}{\bone^\top \YY_{n-1}}\right\|>\delta}\, \Big| \,\XX_{n-1}\right]
\\
&
\qquad
\le
\frac{1}{\delta^\varepsilon m^{1+\varepsilon/2}}\sum_{n=1}^m 
\frac{
E\Big[|M_{n,i}|^{2+\varepsilon}|\XX_{n-1}\Big]
\|\YY_{n-1}\|^{2+\varepsilon}
}
{
(\bone^\top \YY_{n-1})^{2+\varepsilon}
}
\le
\frac{C}{m^{1+\varepsilon/2}}\sum_{n=1}^m \left\| \YY_{n-1}\right\|^{2+\varepsilon}\to 0
\end{split}
\end{equation*}
if $m\to \infty$ as the $(2+\varepsilon)$-th moments are finite.  
This means that the Lindeberg condition is satisfied.
Additionally, 
\begin{equation*}
\begin{split}
\frac{1}{m} &\sum_{n=1}^m E \left[
\left(
 \frac{M_{n,i}\YY_{n-1}^{\top}}{\bone^\top \YY_{n-1}}
\right)
\left(
\frac{M_{n,i}\YY_{i-1}^{\top}}{\bone^\top \YY_{n-1}}
\right)^\top  \,\Big|\, \XX_{n-1}
\right]
=
\frac{1}{m} \sum_{n=1}^m \frac{\YY_{n-1} (\bv_i^\top \YY_{n-1}) \YY_{n-1}^\top }{(\bone^\top \YY_{n-1})^2}\\
&
\to
E\left[
\frac{
\widetilde{\YY}(\bv_i^\top\widetilde{\YY})\widetilde{\YY}^\top 
}{(\bone^\top \widetilde{\YY})^2}
\right] \quad \textrm{a.s.}, \quad m\to \infty,
\end{split}
\end{equation*}
so for any $i=1,\dots, p$ it holds that
\[
\sqrt{m}\left[\widehat{\bmu}_{m,i}^\wcls-\bmu_i\right]\DA \cN(0,\Sigma), \quad m\to \infty,
\]
where
\[
\Sigma=\left( E\left[\frac{\widetilde{\YY}\widetilde{\YY}^\top}{1^\top \widetilde{\YY}}\right]\right)^{-1}E\left(
\frac{
\widetilde{\YY}\left(
\bv_i^\top\widetilde{\YY}
\right)
\widetilde{\YY}^\top }
{
(\bone^\top \,\widetilde{\YY})^2
}
\right)
\left( E\left[\frac{\widetilde{\YY}\widetilde{\YY}^\top}{1^\top \widetilde{\YY}}\right]\right)^{-1}.
\]

Let us discuss the cases of the $\cls$ and the $\wcls$ estimators of the matrix $\bV$. First, based on the formula (\ref{sigma}) we have
\begin{equation}\label{felb1}
\begin{split}
&\sqrt{m}\left[
\widehat{\bv}_{m,i}^\cls
-\bv_i
\right]^\top\\
&
=
\frac{1}{
\sqrt{m} }
\left[
\sum_{n=1}^m \left( N_{n,i}+\left(
\left(\widehat{\bmu}_{m,i}^\cls-\bmu_i\right)^\top\YY_{n-1} \right)^2
+
2M_{n,i}\left(\widehat{\bmu}_{m,i}^\cls-\bmu_i\right)^\top \YY_{n-1}
\right) \YY_{n-1}^\top 
\right]\\
& \quad \times
\left[
\frac{1}{m}
\sum_{n=1}^m \YY_{n-1}\YY_{n-1}^\top 
\right]^{-1}
=:
\frac{1}{\sqrt{m}} \left[A_1+A_2+A_3\right] \left[\frac{1}{m}\sum_{n=1}^m\YY_{n-1}\YY_{n-1}^\top \right]^{-1},
\end{split}
\end{equation}
where as the second moments are finite by ergodicity $\left[\sum_{n=1}^m\YY_{n-1}\YY_{n-1}^\top /m\right]^{-1}\to E[\widetilde{\YY}\widetilde{\YY}^\top]^{-1}$ almost surely, as $m\to \infty$.
By the previous parts of the proof one can easily see that $A_2/\sqrt{m}=o_P(1)$ and $A_3\sqrt{m}=o_P(1)$, as $m\to \infty$ if for some $\varepsilon>0$ the $(4+2\varepsilon)$-th moments of the number of offsprings and innovations all exist. 
We apply the Martingal Central Limit Theorem to the sequence
$
\sum_{n=1}^m N_{n,i}\YY_{n-1}^\top/{\sqrt{m}}$, $m=1,2,\dots$
It is clear that 
\[
\sum_{n=1}^mE\left[ \frac{N_{n,i}\YY_{n-1}^\top}{\sqrt{m}} \, \Big| \, \YY_{n-1}\right]=\bnull, \qquad m=1,2,\dots
\]
As the sixth moments are finite ergodicity results that
\begin{equation*}
\begin{split}
&\frac{1}{m}\sum_{n=1}^m\! \YY_{n-1}^\top E[N_{n,i}^2|\YY_{n-1}]\YY_{n-1}\!=\!
\frac{1}{m}\sum_{n=1}^m \YY_{n-1}^\top \left[ (\Beta_i-3\bv_i^2)^\top \YY_{n-1}\!+\!2\bv_i^\top\YY_{n-1}\YY_{n-1}^\top\bv_i \right]\! \YY_{n-1}
\\
&
\to
E\left[\left[ (\Beta_i-3\bv_i^2)^\top \widetilde{ \YY}+2\bv_i^\top\widetilde{ \YY}\widetilde{ \YY}^\top\bv_i \right]
\widetilde{ \YY}^\top \widetilde{ \YY}\right]
<\infty, \qquad m\to \infty
\end{split}
\end{equation*}
by \ref{fugg3} of Proposition \ref{fugg}.
Let us verify the Lindeberg condition:
\begin{equation*}
\begin{split}
&
\sum_{n=1}^m E
\left[
\left\|
\frac{1}{\sqrt{m}}N_{n,i} \YY_{n-1}^\top 
\right\|^2
\chi_{\big\{\left\|
\frac{1}{\sqrt{m}} N_{n,i}\YY_{n-1}^\top 
\right\|>\delta\big\}} \mid \XX_{n-1}
\right]\\
&
\le
\frac{1}{\delta^\varepsilon}\sum_{n=1}^m E\left[
\left\|
\frac{1}{\sqrt{m}} N_{n,i}\YY_{n-1}^\top 
\right\|^{2+\varepsilon} \mid \XX_{n-1}
\right]\!
\le
\frac{1}{\delta^\varepsilon m^{1+\varepsilon/2}}\sum_{n=1}^mE
\Big[
|N_{n,i}|^{2+\varepsilon}\|\YY_{n-1}\|^{2+\varepsilon}\mid \XX_{n-1}
\Big]\\
&
=
\frac{1}{\delta^\varepsilon m^{1+\varepsilon/2}}\sum_{n=1}^m
E\Big[|N_{n,i}|^{2+\varepsilon} \mid \XX_{n-1}\Big]\|\YY_{n-1}\|^{2+\varepsilon}\le \frac{C}{m^{1+\varepsilon/2}}\sum_{n=1}^m \left\| \YY_{n-1}\right\|^{6+3\varepsilon}\to 0, \quad m\to \infty.
\end{split}
\end{equation*}
For the last step we applied \ref{prop1:3} of Proposition \ref{prop1}. As a consequence, $A_1/\sqrt{m}$ is asymptotically normal so by (\ref{felb1}) and the previous result the proof is complete.

Next, we discuss the $\wcls$ estimator. Based on the formula (\ref{sigma}) it holds that
\begin{equation}\label{felb1}
\begin{split}
&\sqrt{m}\left[
\widehat{\bv}_{m,i}^\wcls
-\bv_i
\right]^\top\\
&
=
\frac{1}{
\sqrt{m} }
\left[
\sum_{n=1}^m \left( \frac{N_{n,i}}{(\bone^\top \YY_{n-1})^2}+\frac{\left(
\big(\widehat{\bmu}_{m,i}^\cls-\bmu_i\big)^\top\YY_{n-1} \right)^2}{(\bone^\top \YY_{n-1})^2}+
\frac{2M_{n,i}\big(\widehat{\bmu}_{m,i}^\cls-\bmu_i\big)^\top\YY_{n-1}}{(\bone^\top \YY_{n-1})^2}
\right) \YY_{n-1}^\top 
\right]\\
& \quad \times
\left[
\frac{1}{m}
\sum_{n=1}^m \frac{\YY_{n-1}\YY_{n-1}^\top }{(\bone^\top \YY_{n-1})^2}
\right]^{-1}
=:
\frac{1}{\sqrt{m}} \left[A_1+A_2+A_3\right] \left[\frac{1}{m}\sum_{n=1}^m\frac{\YY_{n-1}\YY_{n-1}^\top }{(\bone^\top \YY_{n-1})^2}\right]^{-1},
\end{split}
\end{equation}
where by ergodicity 
\[
\left[
\frac{1}{m}\sum_{n=1}^m \frac{\YY_{n-1}\YY_{n-1}^\top}{(\bone^\top \YY_{n-1})^2}\right]^{-1} \to E\left[\frac{\widetilde{\YY}\widetilde{\YY}^\top}{(\bone^\top \widetilde{\YY})^2}
\right]^{-1}, \qquad \textrm{a.s.}, \qquad m\to \infty.
\]
By the previous parts of this proof one can easily prove, that $A_2/\sqrt{m}=o_P(1)$ and $A_3\sqrt{m}=o_P(1)$, as $m\to \infty$, if for some $\varepsilon>0$ the $(2+\varepsilon)$-th moments of the offsprings and innovations all exist. 
Let us apply the Martingal Central Limit Theorem to the sequence $\frac{1}{\sqrt{m}} \sum_{n=1}^m \frac{N_{n,i}\YY_{n-1}^\top }{(\bone^\top \YY_{n-1})^2}$, $\ m=1,2,\dots$
It is clear that 
\[
\sum_{n=1}^mE\left[ \frac{N_{n,i}\YY_{n-1}^\top }{\sqrt{m}(\bone^\top \YY_{n-1}^2)} \, \Big| \,\XX_{n-1}\right]=\bnull, \qquad m=1, 2, \dots
\]
As the fourth moments are finite by ergodicity
\[
\frac{1}{m}\sum_{i=1}^m\frac{\YY_{n-1}^\top \YY_{n-1}}{(\bone^\top \YY_{n-1})^2} E[N_{n,i}^2|\YY_{n-1}]\to E\left[\frac{\widetilde{\YY}^\top \widetilde{\YY}}{(\bone^\top \widetilde{\YY})^2}\left( (\Beta_i-3\bv_i^2)^\top \widetilde{ \YY}+2\bv_i^\top\widetilde{ \YY}\widetilde{ \YY}^\top\bv_i \right)\right]<\infty
\]
almost surely as $m\to \infty$.
Let us verify the Lindeberg condition:
\begin{equation*}
\begin{split}
&
\sum_{i=1}^m E
\left[
\left\|
 \frac{N_{n,i}\YY_{n-1}^\top }{\sqrt{m}(\bone^\top \YY_{n-1})^2}
\right\|^2
\chi_{\left\{\left\|
 \frac{N_{n,i}\YY_{n-1}^\top }{\sqrt{m}(\bone^\top \YY_{n-1})^2}
\right\|>\delta\right\}} \,\Big|\, \XX_{n-1}
\right]\\
&
\le
\frac{1}{\delta^\varepsilon}\sum_{n=1}^m E\left[
\left\|
 \frac{N_{n,i}\YY_{n-1}^\top }{\sqrt{m}(\bone^\top \YY_{n-1})^2}
\right\|^{2+\varepsilon} \,\Big|\, \XX_{n-1}
\right]
\le
\frac{1}{\delta^\varepsilon m^{1+\varepsilon/2}}\sum_{i=1}^m \! E
\left[
\frac{
|N_{n,i}|^{2+\varepsilon}\|\YY_{n-1}\|^{2+\varepsilon}}
{(\bone^\top \YY_{n-1})^{4+2\varepsilon}} \,\Big|\, \XX_{n-1}
\right]\\
&
=
\frac{1}{\delta^\varepsilon m^{1+\varepsilon/2}}\sum_{n=1}^m
\frac{E\left[|N_{n,i}|^{2+\varepsilon} \mid \XX_{n-1}\right]\|\YY_{n-1}\|^{2+\varepsilon}}
{(\bone^\top \YY_{n-1})^{4+2\varepsilon}}\le \frac{C}{m^{1+\varepsilon/2}}\sum_{n=1}^m \left\| \YY_{n-1}\right\|^{2+\varepsilon}\to 0, \quad m\to \infty.
\end{split}
\end{equation*}
For the last step we applied \ref{prop1:3} of Proposition \ref{prop1}. As a consequence, $A_1/\sqrt{m}$ is asymptotically normal so by (\ref{felb1}) and the previous result the proof of this asymptotic normality is also complete.

\ref{eltolas:1}-\ref{eltolas:2}
As the proofs are similar we only show the strong consistency of  $\widehat{\bmu}_m^\wcls$.
By the formula (\ref{sigmaw}) we have that
\begin{equation}\label{konzfelb}
\begin{split}
\widehat{\bmu}_m^{\wcls}-\bmu
&=
\left[
\sum_{n=1}^m
\frac{
(\XX_n- \bmu \YY_{n-1})
\YY_{n-1}^\top }
{\bone^\top \YY_{n-1}}
\right]
\left[
\sum_{n=1}^m 
\frac{
\YY_{n-1}\YY_{n-1}^\top }
{\bone^\top \YY_{n-1}}
\right]^{-1} \\
&
\le 
\left[\frac{1}{m}
\sum_{n=1}^m
(\XX_n- \bmu \YY_{n-1})\bone^\top
\right]\left[ \frac{1}{m}
\sum_{n=1}^m
\frac{
\YY_{n-1}\YY_{n-1}^\top }
{\bone^\top \YY_{n-1}}
\right]^{-1}.
\end{split}
\end{equation}
Let us define
\[
\ZZ_n=
\left[
\begin{array}{c}
\XX_n\\
\YY_{n-1}
\end{array}
\right], \qquad n=1,2,\dots, \qquad
\widetilde{\ZZ}=
\left[
\begin{array}{c}
\widetilde{\XX}_1\\
\widetilde{\YY}_0
\end{array}
\right].
\]
It is easy to see that the process $\ZZ_n$, $n=1,2,\dots$, is also ergodic with invariant distribution $\widetilde{\ZZ}$. Applying the function \[
f
\left(
\left[
\begin{array}{c}
\bx\\
\by
\end{array}
\right]
\right):=(\bx- \bmu \by)\bone^\top, \qquad \bx\in \R^p,\qquad  \by\in \R^{p+1},\]
we have
\begin{equation*}
\begin{split}
&
\frac{1}{m}
\sum_{n=1}^m
(\XX_n- \bmu \YY_{n-1})\bone^\top=\frac{1}{m}\sum_{n=1}^m f(\ZZ_n) 
\to
E\big(f(\widetilde{\ZZ})\big)=
 E\Big((\widetilde{\XX}_1-\bmu \widetilde{\YY}_0)\bone^\top\Big)\\
&=E\Big(E\big(\widetilde{\XX}_1-\bmu \widetilde{\YY}_0 \mid \widetilde{\YY}_0\big) \bone^\top\Big)=\bnull, \quad \rm{a.s.}, \quad m \to \infty.
\end{split}
\end{equation*}
Therefore, by (\ref{konzfelb}) it holds that $\widehat{\bmu}_m^{\wcls}-\bmu\to \bnull$ almost surely if $m\to \infty$.
\end{proof}

\subsection{Limit theorems for the martingale differences}

Let $\MM_n$, $n=1,2,\dots$, be a sequence of arbitrary martingale differences on the state space $\R^p,\, p\in \N$, with respect to some filtration $\cF_n, \  n=1,2,\dots$, meaning that $E[\MM_n|\cF_{n-1}]=\bnull$ holds for any $n=1,2,\dots$ We are going to examine such sequences satisfying the following assumption.

\begin{assumption}\label{assump4}
\begin{enumerate}
\item \label{assump4:1}
For some matrix $\widetilde{\bI}\in \R^{p\times p}$ and every $t>0$ it holds that 
\[
\frac{1}{m} \sum_{n=1}^{\lfloor m t\rfloor}E[\MM_n\MM_n^\top \mid \cF_{n-1}]\to \widetilde{\bI} t \qquad \textrm{a.s.},\qquad m\to \infty.
\]

\item \label{assump4:2}
The Lindeberg condition is satisfied meaning that for any $\delta>0$ we have
\[
\sum_{n=1}^{\lfloor m t\rfloor} E\left[
\left\|
\frac{1}{\sqrt{m}} \MM_n
\right\|^2
\chi_{
\left\{
\left\|
\frac{1}{\sqrt{m}} \MM_n
\right\|>\delta
\right\}
}
\,\Big|\, \cF_{n-1}
\right]
\to 0 \qquad \textrm{a.s.}, \qquad m\to \infty.
\]
\end{enumerate}
\end{assumption}

\begin{proposition}\label{proposition}
If Assumption \ref{assump4} holds then for any $T\in (0,\infty)$ we have that
\begin{equation*}
\begin{split}
\Big[\cX_m^{T}(t)\Big]_{t\in [0,1]}
&
:=
\left[
\frac{\sum_{n=m+1}^{m+ \lfloor tTm\rfloor }\MM_n-\frac{\lfloor tTm\rfloor}{m}\sum_{n=1}^m \MM_n}{
\sqrt{m}\left(1+\frac{\lfloor tTm\rfloor}{m}\right)
\left(
\frac{
\lfloor tTm\rfloor
}
{
m+\lfloor tTm\rfloor
}
\right)^\gamma
}
\right]_{t\in [0,1]}\\
&
\DA
\left[\widetilde{\bI}^{1/2}
\frac{\cW\left(\frac{tT}{1+tT}\right)}{\left(\frac{tT}{1+tT}\right)^\gamma}
\right]_{t\in [0,1]}
=:\left[\cX^{T}(t)\right]_{t\in [0,1]}
\end{split}
\end{equation*}
 in the Skorohod space $\cD^p[0,1]$ as $m\to \infty$,
where $\cW(t), \ t\ge 0$, is a $p$-dimensional standard Wiener process.
\end{proposition}
\begin{proof}
By Assumption \ref{assump4} we can apply the
functional Martingale Central Limit Theorem to the triangular array of variables
$
\left\{
\MM_{1}/\sqrt{m},\dots,\MM_{m(1+T)}/\sqrt{m}
\right\}_{m=1,2,\dots}.
$
For reference on the multidimensional Martingale Central Limit Theorem (MCLT) see e.g. \citet{jacod}, Chapter VIII, Theorem 3.33. 
We get that
\[
\Bigg[\frac{1}{\sqrt{m}} \sum_{n=1}^{\lfloor m (1+T)t\rfloor}\MM_n\Bigg]_{t\in [0,1]}\DA \Big[\widetilde{\bI}^{1/2}\cW\big((1+T)t\big)\Big]_{t\in [0,1]}, \qquad T>0, \qquad m\to \infty,
\]
in $\cD^p[0,1]$.
As a result in case of $t=1/(1+T)$ we get that $\left(\sum_{n=1}^m \MM_n \right)/\sqrt{m}\DA \widetilde{\bI}^{1/2}\cW(1)$ as $m\to \infty$.
This means that 
\begin{equation*}
\begin{split}
&
\left[\frac{1}{\sqrt{m}}
\left( \sum_{n=m+1}^{\lfloor m(1+T)t \rfloor}\MM_n
-
\frac{\lfloor t(1+T)m\rfloor-m}{m}\sum_{n=1}^m \MM_n 
\right)\right]_{t\in \left[\frac{1}{1+T},1\right]}
\\ &
\qquad=
\left[\frac{1}{\sqrt{m}}
\left( \sum_{n=1}^{\lfloor m(1+T)t \rfloor}\MM_n
-
\frac{m+\lfloor t(1+T)m\rfloor-m}{m}\sum_{n=1}^m \MM_n 
\right)\right]_{t\in \left[\frac{1}{1+T},1\right]}\\
&
\qquad 
\DA
\left[
\widetilde{\bI}^{1/2}\Big( \cW((1+T)t)-t(1+T)\cW(1)\Big)
\right]_{t\in \left[\frac{1}{1+T},1\right]}
, \qquad m\to \infty,
\end{split}
\end{equation*}
where the convergence is considered in $\cD^p[1/(1+T),1]$.
By rescaling to the interval $[0,1]$ this implies that
\begin{equation}\label{sztoch1}
\begin{split}
&
\left[
\frac{1}{\sqrt{m}} 
\left(\sum_{n=m+1}^{m+\lfloor tTm\rfloor}\MM_n-
\frac{\lfloor tTm\rfloor}{m}\sum_{n=1}^m \MM_n
\right)
\right]_{t\in[0,1]}
\hspace{-.7cm}
\DA
\left[
\widetilde{\bI}^{1/2}\Big( \cW(1+tT)-(1+tT)\cW(1)\Big)
\right]_{t\in \left[0,1\right]} 
\\&
\DD
\left[
\widetilde{\bI}^{1/2} (1+tT)\cW\left(\frac{tT}{1+tT}\right)\right]_{t\in [0,1]}, \qquad m\to \infty.
\end{split}
\end{equation}
The latter equation can be proved by showing that the covariance functions of the Gaussian processes are the same.
Also, it can be shown by elementary methods that
\begin{equation}\label{sztoch2}
\left[\left(1+\frac{\lfloor tTm\rfloor}{m}\right)\left(\frac{\lfloor tTm\rfloor}{m+\lfloor tTm\rfloor}\right)^\gamma\right]_{t\in[0,1]}
\to
\left[ (1+tT)\left(\frac{tT}{1+tT}\right)^\gamma\right]_{t\in [0,1]}, \qquad m\to \infty,
\end{equation}
in $\cD[0,1]$.
The latter statements, (\ref{sztoch1}) and (\ref{sztoch2}) imply that $[\cX_m^T(t)]_{t\in [0,1]}\DA [\cX^T(t)]_{t\in [0,1]}$ as $m\to \infty$.
\end{proof}
Let us define for any $m\in \N$ the processes
\[
\cY_m(t):=
\frac{
\sum_{n=m+1}^{m+\lfloor tm\rfloor} \MM_n-
\frac{\lfloor tm\rfloor }{m} \sum_{n=1}^m \MM_n}{\sqrt{m} \left(1+\frac{\lfloor tm \rfloor}{m}\right) 
\left(\frac{\lfloor tm\rfloor}{m+\lfloor tm\rfloor}\right)^\gamma}, \qquad
\cY(t):=
\frac{\widetilde{\bI}^{1/2}\cW(\frac{t}{1+t}) }{(\frac{t}{1+t})^\gamma}, \qquad t\ge 0.
\]
\begin{theorem}\label{vegtelen}
If Assumption \ref{assump4} holds then $
\cY_m\DA
\cY$ as $ m\to \infty$
in the Skorohod space $\cD^p[0,\infty)$.
\end{theorem}
\begin{proof}
By Theorem 16.7 of \cite{billingsley} the weak convergence of a process in the Skorohod space $\cD[0,\infty)$ follows if the restriction of the process to the interval $[0,T]$ converges in $\cD[0,T]$ for every $T>0$. By checking the proof one can see that this statement holds for $\cD^p[0,\infty)$ as well.
By Proposition \ref{proposition}
\[
\Big[\cY_m(t)\Big]_{t\in [0,T]}=\Big[\cX_m^T(t/T)\Big]_{t\in [0,T]}
\DA
\Big[\cX^T(t/T)\Big]_{t\in [0,T]}
=\Big[\cY(t)\Big]_{t\in [0,T]}, \quad m \to \infty,
\]
in $\cD^p[0,T]$ for any $T>0$
that completes the proof. 
\end{proof}

\subsection{Proofs of the main results}
In this subsection we prove the main results of the paper. First we show conditions providing Assumption \ref{assump4}.
\begin{proposition}\label{foprop}
The following statements hold under $\cH_0$ and Assumption \ref{assump1}.
\begin{enumerate}
\item\label{alt:1}If for some $\varepsilon>0$ the $(2+\varepsilon)$-th and $(4+\varepsilon)$-th moments of the variables in (\ref{def}) are finite then the series of martingale differences $\MM_n$ and $\VV_n, \ n=1,2,\dots$,  satisfy  Assumption \ref{assump4} with the matrices $\widetilde{\bI}$ and $\widetilde{\bJ}$, respectively.
\item\label{alt:2}If for some $\varepsilon>0$ the $ (1+\varepsilon)$-th and $(2+\varepsilon)$-th moments of the variables in (\ref{def}) are finite then the series of martingale differences $\MM'_n$ and $\VV'_n, \ n=1,2,\dots$,  satisfy  Assumption \ref{assump4} with the matrices $\widetilde{\bI}'$ and $\widetilde{\bJ}'$, respectively.
\end{enumerate}
\end{proposition}
\begin{proof}
\ref{alt:1}
Let us show the proof of the statement concerning $\MM_n,\, n=1,2,\dots$ 
For any $t>0$ we have
\begin{equation*}
\begin{split}
\frac{1}{m} \sum_{n=1}^{\lfloor mt \rfloor} & E(\MM_n\MM_n^\top \mid \XX_{n-1})
=
\frac{\lfloor mt \rfloor}{m} \frac{1}{\lfloor mt \rfloor} \sum_{n=1}^{\lfloor mt \rfloor} E(\MM_n\MM_n^\top  \mid \XX_{n-1})
\to t \widetilde{\bJ}, \qquad  m\to \infty.
\end{split}
\end{equation*}
For any $m\in \N$ it holds that
\[
\sum_{n=1}^m E
\Big(
\left\|
\MM_n/\sqrt{m}
\right\|^2
\chi_{
\|
\frac{1}{\sqrt{m}} \MM_n
\|
>\delta}\, \big| \, \XX_{n-1}
\Big)
\le
\frac{1}{\delta^\varepsilon m^{1+\varepsilon/2}}
\sum_{n=1}^m E\big(
\left\|\MM_n
\right\|^{2+\varepsilon} \, \big| \, \XX_{n-1}
\big),
\]
that converges to $0$ almost surely as  by \ref{prop1:2} of Proposition \ref{prop1} we have that
\[
\lim_{m\to \infty}
\frac{1}{ m}
\sum_{n=1}^m E\big(
\|\MM_n
\|^{2+\varepsilon} \big| \XX_{n-1}
\big)<\infty.
\] 
The rest of the proofs are similar therefore we omit them.
\end{proof}

\begin{lemma}\label{konv}
Suppose that some $d$-dimensional, $d\in \N$, process $\bZ_n, \ n=0,1,\dots$, on the state space $\Z^p$ is ergodic. Let $\widetilde{\bZ}$ denote the variable with the invariant distribution.
If $f$ is a function defined on $\Z^{d}$ satisfying $E(f(\widetilde{\bZ}))<\infty$ and $a_m, \, m=1,2,\dots$ is a non-negative sequence tending to infinity as $m\to \infty$ then
\[
\sup_{k>a_m}\frac{\sum_{n=m+1}^{m+k}\big[f(\bZ_{n-1})\!-\!
E\big(f(\widetilde{\bZ})\big)\big]}{k}\!=\!o_P(1), \ \ \
\sup_{k \ge 1} \frac{\sum_{n=m+1}^{m+k}\big[f(\bZ_{n-1})\!-\!E\big(f(\widetilde{\bZ})\big)\big]}{k}\!=\!O_P(1).
\]
\end{lemma}

\begin{proof}
As the process is ergodic
 \[
\frac{\sum_{i=1}^k \big[f(\bZ_{i-1})-E\big(f(\widetilde{\bZ})\big)\big] }{k} \to 0  \qquad \textrm{a.s.}, \qquad k \to \infty
,\] 
that is equivalent to satisfying
\[
\sup_{k>a_m} \frac{\|\sum_{i=1}^k \big[f(\bZ_{i-1})-E\big(f(\widetilde{\bZ})\big)\big]  \|}{k} \PA 0, \qquad m \to \infty
\]
with any real sequence $a_m\to \infty, \ m\to \infty$. Let us note that
\[
p_{m,\by}(\delta):=P\left( \sup_{k>a_m}\frac{\left\|\sum_{i=m+1}^{m+k} \big[f(\bZ_{i-1})-E\big(f(\widetilde{\bZ})\big)\big]  \right\|}{k}> \delta \,\Big|\, \bZ_m=\by\right)\to 0
\]
as  $m\to \infty$, with  $\delta>0$ and $\by\in \Z^{d}$.
As the process $\bZ_n, \ n\in \N$, converges in distribution there exists a compact set $K_\delta\subset \Z^{d}$ for any $\delta>0$ such that
$\sup_{m\in \N}\sum_{\by\notin K_\delta }P(\bZ_m=\by)<\delta/2$.
Let us note that
\begin{equation}\label{felbontas}
\begin{split}
&P\left(\sup_{k>a_m}\frac{\|\sum_{i=m+1}^{m+k} \big[f(\bZ_{i-1})-E\big(f(\widetilde{\bZ})\big)\big]  \|}{k}> \delta\right)\\
&=\sum_{\by\in K_\delta} p_{m,\by}(\delta)P(\bZ_m=y)+\sum_{\by\notin K_\delta}p_{m,\by}(\delta)P(\bZ_m=\by)=:S_1+\delta/2.
\end{split}
\end{equation}
As $p_{m,\by}\to 0$ if $m\to \infty$ and $K_\delta$ is compact, for big enough $m$ it holds that $\sum_{\by\in K_\delta} p_{m,\by}(\delta)<\delta/2$ meaning that $S_1\le 
 \sum_{\by\in K_\delta} p_{m,\by}(\delta)<\delta/2$ for such $m$. This results that the formula in (\ref{felbontas}) converges to $0$ as $m\to \infty$, that completes the proof of the first statement.

By applying the same alterations we get that for any $c\in \R$
\begin{equation*}
\begin{split}
P(A_{m,c}):&=P\Bigg(\sup_{k \ge 1} \frac{\|\sum_{i=m+1}^{m+k}\big[f(\bZ_{i-1})-E\big(f(\widetilde{\bZ})\big)\big] \|}{k}
>c\Bigg)\\
&
= \sum_{\by\in \Z^{d}} P\left(A_{m,c} \mid \bZ_m= \by\right)P(\bZ_m=\by)
= \sum_{\by\in \Z^{d}} P(A_{0,c} \mid \bZ_0=\by)P(\bZ_m=\by)
.
\end{split}
\end{equation*}
Consider the previously introduced compact set $K_\delta\subset \Z^{d}$ for every $\delta>0$ which satisfies $\sup_{m\in \N}\sum_{\by\notin K_\delta }P(\bZ_m=\by)<\delta/2$. For every $\by\in \Z^{d}$ there is an index $c=c(\by)$ such that $ P( A_{0,c} | \bZ_0=\by)<\delta/2$. 
As 
$K_\delta$ is compact, it has a finite number of points meaning that $c_{K_\delta}:=\max_{\by\in K_\delta} c(\by)$ exists. Then
\begin{equation*}
\begin{split}
&
P\Bigg( \!\sup_{k \ge 1} \frac{\|\sum_{i=m+1}^{m+k}\!\big[f(\bZ_{i-1})\!-\!E\big(f(\widetilde{\bZ})\big)\big]  \|}{k}
\!>\!c_{K_\delta}\!\Bigg)\!\!
\le \!\sum_{\by\in {K_\delta}} \frac{\delta}{2} P(\bZ_m=\by)\!+\!\!\sum_{\by\notin {K_\delta}} 1 P(\bZ_m=\by)\!
\le \! \delta,
\end{split}
\end{equation*}
that completes the proof.
\end{proof}
\begin{proposition}\label{emp}
 If the sequence $\XX_n, \ n\in \N$, satisfies Assumption \ref{assump1}
\begin{enumerate}
\item \label{emp:1} and the $(4+\varepsilon$)-th moments of the variables in (\ref{def}), the number of offsprings and innovations are finite for some $\varepsilon>0$ then 
\[
\sup_{k \ge 1 } \frac{\left\|\sum_{n=m+1}^{m+k} \widehat{\MM}_{m,n}^\cls - \left( \sum_{n=m+1}^{m+k} \MM_n- \frac{k}{m}\sum_{n=1}^m \MM_n\right)\right\|}{g_{\gamma}(m,k)}=o_P(1), \qquad m\to \infty.
\]  
\item\label{emp:2}
If for some $\varepsilon>0$ the $(2+\varepsilon)$-th moments of the variables in (\ref{def}) exist then
\[
\sup_{k \ge 1 } \frac{\left\|\sum_{n=m+1}^{m+k} \widehat{\MM}_{m,n}^\wcls  - \left( \sum_{n=m+1}^{m+k} \MM'_n - \frac{k}{m}\sum_{n=1}^m \MM'_n\right)\right\| }{g_{\gamma}(m,k)}=o_P(1), \qquad m\to \infty.
\]  
\item \label{emp:3}  If for some $\varepsilon>0$   the ($6+\varepsilon$)-th moments of the variables in (\ref{def}) exist then
\[
\sup_{k \ge 1 } \frac{\left\|\sum_{n=m+1}^{m+k} \widehat{\VV}_{m,n}^\cls - \left( \sum_{n=m+1}^{m+k} \VV_n- \frac{k}{m}\sum_{n=1}^m \VV_n\right)\right\|}{g_{\gamma}(m,k)}=o_P(1), \qquad m\to \infty.
\]  
\item\label{emp:4}
If the fourth moments of the variables in (\ref{def}) exist then
\[
\sup_{k \ge 1 } \frac{\left\|\sum_{n=m+1}^{m+k} \widehat{\VV}_{m,n}^\wcls -\left( \sum_{n=m+1}^{m+k} \VV'_n - \frac{k}{m}\sum_{n=1}^m \VV'_n\right)\right\| }{g_{\gamma}(m,k)}=o_P(1), \qquad m\to \infty.
\]  
\end{enumerate}
\end{proposition}
\begin{proof}
Suppose that some $d$-dimensional, $d\in \N$, process $\bZ_n, \ n=0,1,\dots,$ on the state space $\N^d$ is ergodic with invariant distribution $\widetilde{\bZ}$ 
and consider
\[
E_m:=\sup_{k\ge 1} \frac{\|\sum_{n=m+1}^{m+k} \bZ_{n-1} -\frac{k}{m} \sum_{n=1}^m \bZ_{n-1}\|}{\sqrt{m}g_\gamma(m,k)}
\]

We are going to show that $E_m=o_P(1)$ as $m\to \infty$. By defining $\bZ'_n:=\bZ_n-E(\widetilde{\bZ})$ we have that
\[
E_m\le \sup_{k\ge 1} \frac{\|\sum_{n=m+1}^{m+k} \bZ'_{n-1}\|}{\sqrt{m}g_\gamma(m,k)}+\sup_{k\ge 1} \frac{k}{m}\frac{\|\sum_{n=1}^{m} \bZ'_{n-1}\|}{\sqrt{m}g_\gamma(m,k)}=:D_1(m,k)+D_2(m,k).
\]
For some $d>0$ we have the inequalities
\begin{equation}\label{ineq}
g_\gamma(m,k)\ge 
\left\{
\begin{array}{cc}
d \, m^{1/2-\gamma}k^\gamma, & k < m,\\
d \, m^{-1/2} k, & k \ge m.\\
\end{array}
\right. 
\end{equation}
By these bounds and ergodicity
\begin{equation*}
\begin{split}
D_2(m,k) \le \! \sup_{1\le k<m}
\frac{k}{m} \frac{\|\sum_{n=1}^m \bZ'_{n-1}\|}{\sqrt{m} ( d m^{1/2-\gamma} k^\gamma)}
+
 \sup_{m \le k}
\frac{k}{m} \frac{\|\sum_{n=1}^m \bZ'_{n-1}\|}{\sqrt{m} ( d  m^{-1/2} k)}
\le
\frac{2}{d}\frac{\|\sum_{n=1}^m \bZ'_{n-1}\|}{m}=o_P(1)
\end{split}
\end{equation*}
as $m\to \infty$.
Applying Lemma \ref{konv} we get
\begin{equation*}
\begin{split}
&
D_1(m,k) \le \! \sup_{1\le k\le \sqrt{m}}
\frac{\|\sum_{n=m+1}^{m+k} \bZ'_{n-1}\|}{\sqrt{m} ( d m^{1/2-\gamma} k^\gamma)}
+ \!
\sup_{ \sqrt{m}<k\le m}
 \frac{\|\sum_{n=m+1}^{m+k} \bZ'_{n-1}\|}{\sqrt{m} ( d m^{1/2-\gamma} k^\gamma)}
+
 \sup_{m \le k}
 \frac{\|\sum_{n=m+1}^{m+k}\bZ'_{n-1}\|}{\sqrt{m} ( d  m^{-1/2} k)}\\
&
\le \frac{1}{d} \sup_{1\le k\le \sqrt{m}} \left(\frac{k}{m}\right)^{1-\gamma}
\frac{\|\sum_{n=m+1}^{m+k} \bZ'_{n-1}\|}{k}+
\frac{1}{d} \sup_{\sqrt{m}<k\le m} \left(\frac{k}{m}\right)^{1-\gamma}
\frac{\|\sum_{n=m+1}^{m+k} \bZ'_{n-1}\|}{k}\\
&
+ \frac{1}{d}\sup_{m \le k} 
 \frac{\|\sum_{n=m+1}^{m+k}\bZ'_{n-1}\|}{k}
\le
 \frac{1}{d}\frac{1}{\sqrt{m}^{1-\gamma}} \sup_{1\le k} 
\frac{\|\sum_{n=m+1}^{m+k} \bZ'_{n-1}\|}{k}\\
&+
\frac{1}{d} \sup_{\sqrt{m}<k}
\frac{\|\sum_{n=m+1}^{m+k} \bZ'_{n-1}\|}{k}
+ \frac{1}{d}\sup_{m \le k} 
 \frac{\|\sum_{n=m+1}^{m+k}\bZ'_{n-1}\|}{k}=o_P(1), \qquad m\to \infty.
\end{split}
\end{equation*}
\ref{emp:1}
By definition $\MM_{n}=\XX_n-\bmu \YY_{n-1}$ and $\widehat{\MM}_{m,n}^\cls=\XX_n-\widehat{\bmu}_m^\cls \YY_{n-1}$ for any $n=1,2,\dots$ and $\sum_{n=1}^{m} \widehat{\MM}_{m,n}^\cls =0$, therefore
\begin{equation*}
\begin{split}
\sup_{k \ge 1 }& \frac{\left\|\sum_{n=m+1}^{m+k} \widehat{\MM}_{m,n}^\cls  - \left( \sum_{n=m+1}^{m+k} \MM_n - \frac{k}{m}\sum_{n=1}^m \MM_n\right)\right\|}{g_{\gamma}(m,k)}\\
&=\sup_{k \ge 1 } \frac{\left\|\sum_{n=m+1}^{m+k} \left(\widehat{\MM}_{m,n}^\cls  -\MM_n\right) - \frac{k}{m}\sum_{n=1}^m \left(\widehat{\MM}_{m,n}^\cls -\MM_n\right)\right\| }{g_{\gamma}(m,k)}\\
&
\le \sqrt{m} \left\|
\bmu-\widehat{\bmu}_m^\cls
\right\|
\sup_{k \ge 1 } \frac{\left\|\sum_{n=m+1}^{m+k} \YY_{n-1} - \frac{k}{m}\sum_{n=1}^m \YY_{n-1}\right\|  }{\sqrt{m}g_{\gamma}(m,k)}
\end{split}
\end{equation*}
By the remark in the beginning and \ref{normal1:1} of Theorem \ref{normal1} the proof of \ref{emp:1} is complete.

\ref{emp:3}
As
$
\NN_n=\MM_n^2-\bV \YY_{n-1}
$ and 
$\widehat{\NN}_{m,n}^\cls=(\widehat{\MM}_{m,n}^\cls)^2-\widehat{\bV}_m^\cls \YY_{n-1}$, $n=1,2,\dots$,
we have
\begin{equation*}
\begin{split}
&
\sup_{k \ge 1 } \frac{\left\|\sum_{n=m+1}^{m+k} \widehat{\NN}_{m,n}^\cls - \left( \sum_{n=m+1}^{m+k} \NN_n - \frac{k}{m}\sum_{n=1}^m \NN_n\right)\right\|}{g_{\gamma}(m,k)}\\
&
=
\sup_{k \ge 1 }
\frac{\left\|
\sum_{n=m+1}^{m+k} 
\big( (\widehat{\MM}_{m,n}^\cls)^2-\MM_n^2\big) -
 \frac{k}{m} \sum_{n=1}^m  \big( (\widehat{\MM}_{m,n}^\cls)^2-\MM_n^2\big)
\right\|}{g_{\gamma}(m,k)}\\
&
+\sqrt{m}[\bV-\widehat{\bV}_m^\cls]\sup_{k\ge1} \frac{\|\sum_{n=m+1}^{m+k} \YY_{n-1}-\frac{k}{m} \sum_{n=1}^m \YY_{n-1}\|}{\sqrt{m}g_\gamma(m,k)}=:B_1(m,k)+B_2(m,k).
\end{split}
\end{equation*}
By \ref{normal1:1} of Theorem \ref{normal1} we have that $B_2(m,k)=o_P(1)$ as $m\to \infty$.
As for any $i=1,\dots,p$
\begin{equation*}
\begin{split}
\big(\widehat{M}_{m,n,i}^\cls \big)^2-M_{n,i}^2&=
\left(\widehat{\bmu}_{m,i}^\cls-\bmu_i\right)^\top\YY_{n-1}\YY_{n-1}^\top (\widehat{\bmu}_{m,i}^\cls) +\bmu_i^\top \YY_{n-1}\YY_{n-1}^\top \left(\widehat{\bmu}_{m,i}^\cls-\bmu_i\right)\\
&-2\left(\widehat{\bmu}_{m,i}^\cls-\bmu_i\right)^\top X_{n,i}\YY_{n-1},
\end{split}
\end{equation*}
Applying that $\sqrt{m}(\widehat{\bmu}_{m,i}^\cls-\bmu_i)$ is asymptotically normal and term by term using the remark in the beginning we have that $B_1(m,k)=o_P(1)$. We detail the ergodicity of the last term. 
Let us define
\[
\ZZ_n=
\left[
\begin{array}{c}
\XX_n\\
\YY_{n-1}
\end{array}
\right], \qquad n=1,2,\dots, \qquad
\widetilde{\ZZ}=
\left[
\begin{array}{c}
\widetilde{\XX}_1\\
\widetilde{\YY}_0
\end{array}
\right].
\]
The process $\ZZ_n$, $n=1,2,\dots$, is also ergodic with invariant distribution $\widetilde{\ZZ}$ so applying the function 
$
f(\ZZ_n):=X_{n,i}\YY_{n-1}
$, $i=1,\dots, p$
we have
\begin{equation*}
\begin{split}
&
\frac{1}{m}
\sum_{n=1}^m
X_{n,i}\YY_{n-1}=\frac{1}{m}\sum_{n=1}^m f(\ZZ_n) 
\to
E\big(f(\widetilde{\ZZ})\big)=
 E(\widetilde{X}_{1,i}\widetilde{\YY}_0), \quad m \to \infty.
\end{split}
\end{equation*}
The proofs of \ref{emp:2} and \ref{emp:4} are similar, therefore we omit them.
\end{proof}
\begin{proof}[Proof of Theorem \ref{fotetel1}]
Theorem \ref{fotetel1} immediately follows from Theorem \ref{vegtelen}, Proposition \ref{foprop}, and Proposition \ref{emp}.
\end{proof}

\subsection{Alternative hypothesis}
\begin{proof}[Proof of Theorem \ref{alt}]
By definition we have the decomposition
\begin{equation}\label{alt1}
\sum_{n=m+1}^{m+k} \widehat{\MM}_{m,n}^\cls
=\sum_{n=m+1}^{m+k^*-1} \widehat{\MM}_{m,n}^\cls+\sum_{n=m+k^*}^{m+k}(\XX_n-\bmu_0 \YY_{n-1})+\sum_{n=m+k^*}^{m+k} (\bmu_0-\widehat{\bmu}_m^\cls)\YY_{n-1}.
\end{equation} Let us fix $k=2(m+k^*)$.
Applying Theorem \ref{fotetel1} we get that
\[
\frac{\|\sum_{n=m+1}^{m+k^*-1} \widehat{\MM}_{m,n}\|}{g_\gamma(m,k^*-1)}\frac{g_\gamma(m,k^*-1)}{g_\gamma(m,k)}\le\frac{\|\sum_{n=m+1}^{m+k^*-1} \widehat{\MM}_{m,n}\|}{g_\gamma(m,k^*-1)}=O_P(1).
\]
Applying \ref{normal1:1} of Theorem \ref{normal1}, ergodicity, and (\ref{ineq}) we have that with some $d>0$ constant
\[
\frac{\|\sum_{n=m+k^*}^{m+k} (\bmu_0-\widehat{\bmu}_m^\cls)\YY_{n-1}\|}{g_\gamma(m,k)}
\le
\!\sqrt{m}\|\bmu_0-\widehat{\bmu}_m^\cls\|\frac{1}{d}\frac{\|\sum_{n=m+k^*}^{m+k}\YY_{n-1}\|}{k-k^*+1}\frac{k-k^*+1}{k}\!=\!O_P(1).
\] Ergodicity and simple calculations lead to
\begin{equation*}
\begin{split}
&
\frac{\sum_{n=m+k^*}^{m+k} \XX_n-\bmu_0 \YY_{n-1}}{g_\gamma(m,k)}=
\frac{(k-k^*+1)\big(E(\widetilde{\XX}-\bmu_0 \widetilde{\YY})+o_P(1)\big)}{g_\gamma(m,k)}\\
&
\ge \frac{(k-k^*+1)\sqrt{m}}{2k}E(\widetilde{\XX}-\bmu_0 \widetilde{\YY})+o_P(\sqrt{m})
 \ge \frac{\sqrt{m}}{4}E(\widetilde{\XX}-\bmu_0 \widetilde{\YY})+o_P(\sqrt{m}), \quad m\to \infty.
\end{split}
\end{equation*}
As $E(\widetilde{\XX}-\bmu_0 \widetilde{\YY})\neq 0$ putting together the last three computations we get that
\[
\sup_{k\ge 1}  \frac{\|\sum_{n=m+1}^{m+k} \widehat{\MM}_{m,n}^\cls\|}{g_\gamma(m,k)} \ge \frac{\|\sum_{n=m+1}^{m+2(m+k^*)} \widehat{\MM}_{m,n}^\cls\|}{g_\gamma(m,2(m+k^*))}\PA \infty, \qquad m\to \infty,
\]
that completes the proof.
\end{proof}
\begin{proof}[Proof of Proposition \ref{alt3}]
By the decomposition in (\ref{alt1}) and ergodicity we have that
\[
\sup_{1\le k\le N} \frac{\left\|
 \sum_{n=m+1}^{m+k} \widehat{\MM}_{m,n}^\cls
\right\|}{g_\gamma(m,k)}
=O_P(1)+ \sup_{k^* < k\le N}\frac{(k-k^*+1)\big( E(\widetilde{\XX}-\bmu_0 \widetilde{\YY})+o_P(1)\big)}{\sqrt{m}\left(1+\frac{k}{m}\right)\left(\frac{k}{m+k}\right)^\gamma}
\] 
for any $N\in \N$ as $m\to \infty$.
Let us choose $N\in \N$ such that with a constant $C\in \N$ it holds that
\[
N-k^*=\left\{
\begin{array}{ll}
Cm^{(1-2\gamma)/(2-2\gamma)}, & \quad 0 \le b<\frac{1-2\gamma}{2-2\gamma}\\
Cm^{1/2-\gamma(1-b)}, & \quad \frac{1-2\gamma}{2-2\gamma}\le b <1\\
Cm^{b-1/2}, & \quad 1 \le b<\infty\\
\end{array}
\right..
\]
Then it is easy to verify
\[
\lim_{m\to \infty} \sup_{k^*\le k\le N} \frac{k-k^*+1}{\sqrt{m}\left(1+\frac{k}{m}\right)\left(\frac{k}{m+k}\right)^\gamma}\ge 
\left\{
\begin{array}{ll}
C^{1-\gamma}, & \quad 0 \le b<\frac{1-2\gamma}{2-2\gamma}\\
C\left(\theta+C\1_{\left\{b=\frac{1-2\gamma}{2-2\gamma}\right\}}\right)^{-\gamma}, & \quad \frac{1-2\gamma}{2-2\gamma}\le b <1\\
C
\frac{\left(\1_{\{b=1\}}+\theta\right)^{\gamma-1}}{\theta^\gamma}
, & \quad 1 \le b<\infty\\
\end{array}
\right..
\]
that converges to infinity as $C\to \infty$.
We show the inequalities separately in the three cases.
Let $a:=(1-2\gamma)/(2-2\gamma)<1/2$ and note that
\[
 \sup_{k^*\le k\le N} \frac{k-k^*+1}{\sqrt{m}\left(1+\frac{k}{m}\right)\left(\frac{k}{m+k}\right)^\gamma}\ge \frac{N-k^*}{\sqrt{m}\left(1+\frac{N}{m}\right)\left(\frac{N}{m+N}\right)^\gamma}=\frac{\sqrt{m}(N-k^*)}{(N+m)^{1-\gamma} N^\gamma}.
\]
By the definition of $N$ and the form of $k^*$ we know that
\[
N=k^*+(N-k^*)=\lfloor \theta m^b\rfloor +(N-k^*).
\]
The following convergences hold if $m\to \infty$.
\begin{enumerate}
\item
In the first case when {$0\le b<a$} we have
\begin{equation*}
\begin{split}&
\frac{\sqrt{m}(N-k^*)}{(N+m)^{1-\gamma} N^\gamma}=\frac{Cm^{a+1/2}}{(\lfloor \theta m^b\rfloor+Cm^a+m)^{1-\gamma}(\lfloor \theta m^b\rfloor+Cm^a)^\gamma}\\
&
=\frac{Cm^{a+1/2}}{m^{1-\gamma}\left(
\frac{\lfloor \theta m^b\rfloor}{m}+Cm^{a-1}+1
\right)^{1-\gamma} m^{a\gamma} 
\left(
\frac{\lfloor \theta m^b\rfloor}{m^a}+C
\right)^\gamma
}
=
\frac{Cm^{a+1/2-(1-\gamma)-a\gamma}}{(o(1)+C)^\gamma}\to \frac{C}{C^\gamma},
\end{split}
\end{equation*}
as we can easily see that
\[
a+1/2-(1-\gamma)-a\gamma=a(1-\gamma)+\gamma-1/2=\frac{1-2\gamma}{2-2\gamma}(1-\gamma)+\gamma-1/2=0.
\]
\item Secondly, when {$a\le b<1$} we get that
\begin{equation*}
\begin{split}&
\frac{\sqrt{m}(N-k^*)}{(N+m)^{1-\gamma} N^\gamma}
=\frac{Cm^{1-\gamma(1-b)}}
{m^{1-\gamma}\!\left(
\frac{\lfloor \theta m^b\rfloor}{m}+Cm^{-\frac{1}{2}-\gamma(1-b)}+1
\right)^{1-\gamma} \!\! m^{b\gamma}\!
\left(
\frac{\lfloor \theta m^b\rfloor}{m^b}+C m^{\frac{1}{2}-\gamma(1-b)-b}
\right)^\gamma
}\\
&
=
\frac{Cm^{1-\gamma(1-b)-(1-\gamma)-b\gamma}}{\left(\theta+C\1_{\{b=a\}}\right)^\gamma (1+o(1))}\to \frac{C}{(\theta+C\1_{\{b=a\}})^\gamma},
\end{split}
\end{equation*}
as the exponent 
$
1/2-\gamma(1-b)-b
$
is decreasing in $b$ and for $b=a$ it is exactly $0$.
\item Finally, if {$1 \le b<\infty$} then
\begin{equation*}
\begin{split}&
\frac{\sqrt{m}(N-k^*)}{(N+m)^{1-\gamma} N^\gamma}
=\frac{Cm^{b}}
{m^{b(1-\gamma)}\left(
\frac{\lfloor \theta m^b\rfloor}{m^b}+Cm^{-\frac{1}{2}}+m^{1-b}
\right)^{1-\gamma} m^{b\gamma} 
\left(
\frac{\lfloor \theta m^b\rfloor}{m^b}+C m^{-\frac{1}{2}}
\right)^\gamma
}\\
&
=
\frac{C}{\left(\theta+\1_{\{b=1\}}\right)^{1-\gamma}\theta^\gamma(1+o(1))}\to \frac{C}{\left(\theta+\1_{\{b=1\}}\right)^{1-\gamma}\theta^\gamma}.
\end{split}
\end{equation*}
\end{enumerate}
This completes the proof.
\end{proof}

\begin{proof}[Proof of Proposition \ref{alt4}] \ref{alt4:1} As $b<1$, for every $d\in \R_+$ and large enough $m\in \N$ it holds that
\[
0\le P(\tau_{m,1}^\cls < k^*)\le  P(\tau_{m,1}^\cls < d \, m )\to 1-F\left(\left(\frac{1+d}{d}\right)^{1/2-\gamma}x_\alpha\right), \qquad m\to \infty,
\]
where $F$ is the distribution function of $\sup_{0\le t \le 1} \|\cW(t)\|/t^\gamma$.
As for every $\delta>0$ there is a $d>0$ such that $1-F(((1+d)/d)^{1/2-\gamma}x_\alpha)<\delta$ this completes the proof of \ref{alt4:1}.

\ref{alt4:2}
For $b=1$ we have that 
\begin{equation*}
\begin{split}
&P(\tau_{m,1}^\cls<k^*)=
 P\left(\sup_{0\le t \le \theta}\frac{\left\|
\left(\widehat{\bI}_m^\cls\right)^{-1/2} \sum_{n=m+1}^{m+\lfloor mt\rfloor}\widehat{\MM}_{m,n}^\cls
\right\|}{\sqrt{m}\big(1+\frac{\lfloor mt\rfloor}{m}\big)\big( \frac{\lfloor mt\rfloor}{m+\lfloor mt\rfloor}\big)^\gamma}>x_\alpha\right)\\
&
\to
P\left(\left(
\frac{\theta}{1+\theta}
\right)^{1/2-\gamma}\!\!\! \sup_{0\le t \le 1}\frac{\left\| \cW(t) \right\|}{t^\gamma} > x_\alpha\right)
=
1-F\left( \left(
\frac{1+\theta}{\theta}
\right)^{1/2-\gamma}x_\alpha\right)\in (0,\alpha], \quad m\to \infty.
\end{split}
\end{equation*}

If $b>1$, then for every $d\in \R_+$ there is an $m\in \N$ such that $m^{b-1}\theta>d$ meaning that for large enough $m$ we have
\[
P(\tau_{m,1}^\cls<dm)\le P(\tau_{m,1}^\cls<k^*) \le P(\tau_{m,1}^\cls<\infty)\to \alpha, \qquad m\to \infty.
\]
As we have previously seen, for every $d\in R_+$ it holds that
\[
\lim_{m\to \infty}P(\tau_{m,1}^\cls<dm)=1-F\left( \left(
\frac{1+d}{d}
\right)^{1/2-\gamma}x_\alpha\right), \qquad m\to \infty.
\]
Therefore, if $b>1$ then
$
\lim_{m\to \infty}P(\tau_{m,1}^\cls<k^*)=\alpha$ as $m\to \infty$, and this completes the proof.
\end{proof}

\vspace{1cm}\noindent
Fanni Ned\'enyi\\
Bolyai Institute, University of Szeged\\
Aradi v\'ertan\'uk tere 1, 6720 Szeged, Hungary\\
\end{document}